\documentclass[12pt, english]{article}
\usepackage[utf8]{inputenc}
\usepackage{amsthm}
\usepackage{amssymb}
\usepackage[all]{xy}
\usepackage{amsmath}
\usepackage{tikz-cd}
\usepackage{tikz}
\usepackage{float}
\usepackage{bbm}
\usepackage{bm}

\tikzcdset{scale cd/.style={every label/.append style={scale=#1},
		cells={nodes={scale=#1}}}}
\usepackage{rotating}
\usepackage{caption, booktabs}
\usepackage{verbatim}
\usetikzlibrary{calc}

\usepackage{tikz}
\usetikzlibrary{patterns}

\usepackage{pdflscape}

\usepackage{mathtools}
\usepackage[english]{babel}

\usepackage{hyperref}
\usepackage{bookmark}

\newtheorem{definition}{Definition}[section]

\newtheorem{proposition}[definition]{Proposition}
\newtheorem{remark}[definition]{Remark}
\newtheorem{lemma}[definition]{Lemma}

\newtheorem{theorem}[definition]{Theorem}
\newtheorem{corollar}[definition]{Corollary}
\newtheorem{construction}[definition]{Construction}

\newtheorem{example}[definition]{Example}

\DeclareMathOperator{\Hom}{Hom}

\DeclareMathOperator{\Lie}{Lie}
\DeclareMathOperator{\Aut}{Aut}

\DeclareMathOperator{\Ext}{Ext}

\DeclareMathOperator{\Cone}{Cone}

\DeclareMathOperator{\Spec}{Spec}

\DeclareMathOperator{\codim}{codim}
\DeclareMathOperator{\Supp}{Supp}

\DeclareMathOperator{\Min}{Min}
\DeclareMathOperator{\Convhull}{Convhull}

\newcommand{\uproman}[1]{\uppercase\expandafter{\romannumeral#1}}

\newcommand{\RNum}[1]{\uppercase\expandafter{\romannumeral #1\relax}}

\begin{document}
	
	\title{The Kodaira-Spencer map for minimal toric hypersurfaces}
	\author{Julius Giesler \\ University of T\"ubingen}
	\date{\today}
	\maketitle
	\begin{abstract}
		In this article we study infinitesimal deformations of toric hypersurfaces. We introduce a Kodaira-Spencer map and compute its kernel. By introducing some new Laurent polynomials we make our computation as explicit as possible. This widely generalizes results of Griffiths for projective hypersurfaces.
	\end{abstract}
	%	\tableofcontents
	
	%This generalizes the known characterization of the kernel of $\kappa_f$ for smooth degree $d$ projective hypersurfaces via the jacobian ideal of $f$ and our computation of $\ker(\kappa_f)$ is still that explicit as for smooth projective hypersurfaces.

	\section{Introduction}
	We start with an $n$-dimensional lattice polytope $\Delta$ and a nondegenerate Laurent polynomial $f$ with Newton polytope $\Delta$ and set
	\begin{align*}
		Z_f := \{ f = 0 \} \subset T.
	\end{align*}
	%We vary $f$ over all nondegenerate Laurent polynomials $U_{reg}(\Delta)$ with fixed $n$-dimensional Newton polytope $\Delta$. We associate to $\Delta$ its \textit{Fine interior} $F(\Delta)$ by moving all hyperplanes touching $\Delta$ one step into the interior. 
	By results from (\cite{Bat22}) there is (under a mild condition $F(\Delta) \neq \emptyset$) a projective toric variety $\mathbb{P}$ to a simplicial fan $\Sigma$, such that the closure $Y_f$ of $Z_f$ in $\mathbb{P}$ has at most terminal singularities and the canonical divisor $K_{Y_f}$ is nef. We say that $Y_f$ is a \textit{minimal model of } $Z_{f}$. \\
	The toric variety $\mathbb{P}$ does not depend on $f$. Therefore we vary $f$ over all such polynomials, write $f \in B := B(\Delta)$ to obtain a family of minimal toric hypersurfaces
	\[ \mathcal{X}:= \{(x,f) \in \mathbb{P} \times B \mid \, x \in Y_f \} \overset{pr_2}{\rightarrow}  B. \]
%	Note the special case that $\mathbb{P} = \mathbb{P}^n$ and $B \subset |\mathcal{O}_{\mathbb{P}^n}(d)|$ is the open subset of all $f$ such that $Z_f$ is smooth and $Y_f$ intersects the strata of $\mathbb{P}^n$ transversely (in particular $Y_f$ is smooth). In this case $\Delta = d \cdot \Delta_n$ equals $d$-times the standard simplex and the condition $F(\Delta) \neq \emptyset$ just means that $d \geq n+1$. \\
	Next to a tangent vector 
	\[ \Spec \, \mathbb{C}[\epsilon]/(\epsilon^2) \rightarrow B  \]
	at $f \in B$ we associate a first-order infinitesimal deformation of $Y_f$ (in $\mathcal{X}$) by base change along $v$. These infinitesimal deformations and all those in $\mathbb{P}$ are parametrized by \textit{Kodaira Spencer maps}
\begin{align*}
	& \kappa_{\mathbb{P},f}: H^0(Y,N_{Y/\mathbb{P}}) \rightarrow \Ext_{\mathcal{O}_{Y}}^1(\Omega_Y^1, \mathcal{O}_Y), \\
	& \kappa_f: \, \,  H^0(Y,N_{Y/\mathcal{X}}) \rightarrow \Ext_{\mathcal{O}_{Y}}^1(\Omega_Y^1, \mathcal{O}_Y).
\end{align*}
In fact $H^0(Y,N_{Y/\mathcal{X}}) \subset H^0(Y,N_{Y/\mathbb{P}})$ and $\kappa_f$ equals the restriction of $\kappa_{\mathbb{P},f}$ (see section \ref{section_number_of_mod_general_case}). In Theorem \ref{theorem_number_of_moduli} we identify the kernel of $\kappa_{\mathbb{P},f}$ with the Lie algebra of the automorphism group of $\mathbb{P}$:
\begin{theorem} \label{theorem_number_of_moduli_intro}
	Let $\Delta$ be an $n$-dimensional lattice polytope, where $n \geq 2$, with $l^*(\Delta) > 0$. If $n=2$ assume that $l^*(\Delta) \geq 2$. Then given $f \in B$
	\begin{align} 
		\ker(\kappa_{\mathbb{P},f}) \cong  \Lie \, \Aut(\mathbb{P})
	\end{align} 		
\end{theorem}
where $l^*(\Delta)$ denotes the number of interior lattice points of $\Delta$. The proof is cohomological. Concerning the cokernel of $\kappa_{\mathbb{P},f}$ we prove:
\begin{corollar} \label{corollar_all_inf_def_come_from..._intro}
	Given the conditions of the theorem if $n \geq 4$ then the following sequence is exact
	\begin{align*}
		0 \rightarrow Im(\kappa_{\mathbb{P},f}) \rightarrow \Ext_{\mathcal{O}_{Y}}^1(\Omega_{Y}^1, \mathcal{O}_Y) \rightarrow \Ext_{\mathcal{O}_{\mathbb{P}}}^1(\Omega_{\mathbb{P}}^1, \mathcal{O}_{\mathbb{P}}) \rightarrow 0 
	\end{align*}
	that is all infinitesimal deformations of $Y_f$ arise from $Im(\kappa_{\mathbb{P},f})$ or from infinitesimal deformations of $\mathbb{P}$.
\end{corollar}
A more detailed description of the second term  is a topic of current research (see \cite{IlTu18}, \cite{Pet20}) and (\cite{Mav03}) for the case that $Y_f$ is Calabi-Yau.
\leavevmode
\\ \\
To be more explicit on $\ker(\kappa_{\mathbb{P},f})$ we recall the definition of the \textit{canonical closure} $C(\Delta)$ of $\Delta$, a polytope that might be slightly larger than $\Delta$. Given $P \subset M_{\mathbb{R}}$ let $L(P)$ be the $\mathbb{C}$-vector space with basis $x^m$, $m \in P \cap M$. Then
\begin{align*}
	& H^0(Y,N_{Y/\mathcal{X}}) \cong L(\Delta) / \mathbb{C} \cdot f \\
	& H^0(Y,N_{Y/\mathbb{P}}) \cong L(C(\Delta)) / \mathbb{C} \cdot f.
\end{align*}

We apply results from (\cite{BG99}) to identify a basis of $\ker(\kappa_{\mathbb{P},f})$ explicitly with certain Laurent polynomials having support on $C(\Delta)$ (Corollary \ref{theorem_basis_ker_kappa_P_subset_LCDelta}):
\begin{corollar} \label{theorem_basis_ker_kappa_P_subset_LCDelta_intro}
	Given the conditions of Theorem \ref{theorem_number_of_moduli_intro}, let 
	\begin{align*}
		f = \sum\limits_{m \in M \cap \Delta} a_m x^m \in B
	\end{align*}
	Then
	\begin{align*}
	\ker(\kappa_{\mathbb{P},f}) \cong \Big\langle \, x_i \cdot \frac{\partial f}{\partial x_i} \mid i=1,...,n \, \Big\rangle \quad \oplus \bigoplus\limits_{\alpha \in R(N,\Sigma)} w_{-\alpha}(f) \cdot \mathbb{C} 
	\end{align*}
\end{corollar}
where $R(N,\Sigma) \subset M$ denote the \textit{(Demazure) roots} of $\Sigma$ and $w_{-\alpha}(f)$ denote certain Laurent polynomials. Concerning the kernel of $\kappa_f$ we prove
\begin{theorem} \label{theorem_numb_mod_not_can_closed_intro}
	Given the conditions of Theorem \ref{theorem_number_of_moduli_intro} then
	\begin{align*}
	\ker(\kappa_{f}) \cong \Big\langle \, x_i \cdot \frac{\partial f}{\partial x_i} \mid i=1,...,n \, \Big\rangle \quad \oplus \bigoplus\limits_{\alpha \in R(N,\Sigma_{\Delta})} w_{-\alpha}(f) \cdot \mathbb{C}
\end{align*}
\end{theorem}

%Here $R(N,\Sigma_{\Delta})$ are the roots of the normal fan $\Sigma_{\Delta}$ of $\Delta$ and we always have
%\[ R(N,\Sigma_{\Delta}) \subset R(N,\Sigma_{C(\Delta)}).  \]
%This Theorem gives very precise information on $\ker(\kappa_f)$.
%\leavevmode \\
%the dimension $\dim \, Im(\kappa_f)$ is also called the number of moduli of $\mathcal{X} \rightarrow B$. 

\textit{Application:} Let $n=3$. Then $Y$ is smooth and thus 
\begin{align}
\Ext_{\mathcal{O}_{Y}}^1(\Omega_Y^1, \mathcal{O}_Y) \cong H^1(Y,T_Y).
\end{align}
There is a period map
\begin{align*}
\mathcal{P}_{B,f}: f \mapsto [H^{2,0}(Y_f)].
\end{align*}
This map is holomorphic and by a result of Griffiths the differential $d \mathcal{P}_{B,f}$ factors as
\[
\begin{tikzcd}
L(C(\Delta))/ \mathbb{C} \cdot f \arrow{r}{\kappa_{\mathbb{P},f}} \arrow[swap]{dr}{d \mathcal{P}_{B,f}} & H^1(Y,T_Y) \arrow{d}{\Phi_f} \\
&  \Hom(H^{2,0}(Y_f), H^{1,1}(Y_f))
\end{tikzcd}
\]
The infinitesimal Torelli theorem for $Y_f$ asks if $\Phi_f$ is injective. If we know $\ker(d \mathcal{P}_{B,f})$ we may check if
\begin{align}\label{equation_*}
\ker(d \mathcal{P}_{B,f}) \supseteq \ker(\kappa_f)
\end{align}
is an equality, in which case we get 
\begin{align*}
\ker(\Phi_{f \vert{Im \, \kappa_{\mathbb{P},f}}}) = 0
\end{align*}
giving a partial result towards the ITT. A calculation of $\ker(d \mathcal{P}_{B,f})$ should work with jacobian rings. In the dissertation (\cite{Gie23}) we worked this out for $n=3$. In fact these methods should also apply to higher dimensions.

\begin{comment}

We compute
\begin{proposition}
Given the conditions of the theorem...
\begin{align*}
\ker(\kappa_f) \cong \langle x_i \cdot \frac{\partial f}{\partial x_i}, \quad i=1,...,n, \, w_{-\alpha}(f) \mid \alpha \in R(N,\Sigma_{\Delta}) \rangle
\end{align*}
\end{proposition}
As $\kappa_f$ has a more geometric background than $\kappa_{\mathbb{P},f}$ the dimension $\dim \, Im(\kappa_f)$ is also called the number of moduli of $\mathcal{X} \rightarrow B$.

%\begin{align*}
%	\rho_{min}(Y) = \dim \, \iota^* H^2(\mathbb{P}, \mathbb{Q}) + \dim \, \coker(R_f^1 \otimes R_{Int,f})
%\end{align*}

the infinitesimal Torelli theorem several articles (BrGr,...) and a lot of thoughts of us gives us very strong indices that the map $d \mathcal{P}_{B,f}$ may be used not just for the ITT but also to calculate the Picard number $\rho(Y_f)$ for $f \in B$ outside a countable union of proper subsets (the Noether-Lefschetz locus). 
This property should get inherited to subfamilies of $\mathcal{X} \rightarrow B$, that is given
\[
\begin{tikzcd}
\mathcal{X}^0 \arrow{d}{\subseteq} \arrow{r}{pr_{2 \vert{\mathcal{X}^0}}} & B^0 \arrow{d}{\subseteq} \\
\mathcal{X} \arrow{r}{pr_2} & B
\end{tikzcd}
\]
the Picard number $\rho(Y_f)$ for $f \in B^0$ very general should be computable from $d \mathcal{P}_{B^0,f}$. 
\end{comment}

One might be also interested in subfamilies of $\mathcal{X} \rightarrow B$: Given $B^0 \subseteq B$ with restriction $\kappa_f^0$ of $\kappa_f$ it is obvious that
\begin{align*}
x_1 \cdot \frac{\partial f}{\partial x_1},...,x_n \cdot \frac{\partial f}{\partial x_n} \in \ker(\kappa_f^0)
\end{align*}
but hard to say when $w_{-\alpha}(f) \in \ker(\kappa_f^0)$. We show that if $B^0$ \textit{equals} the set of vertices of $\Delta$ then under some additional condition on the polytope $\Delta$ none of the $w_{-\alpha}(f)$'s contribute to $\ker(\kappa_f)$ (see Lemma \ref{lemma_number_of_moduli_subfam_to_vertices}).

\begin{comment}
to: If $\Delta$ is an $n$-dimensional polytope with $F(\Delta) \neq \emptyset$ and given a facet $\Gamma = \Delta \cap \{ x \in M_{\mathbb{R}} \mid \langle x, n_{\Gamma} \rangle = b_{\Gamma} \}$ of $\Delta$ assume there is no vertex of $\Delta$ with distance $1$ to $\Gamma$, that is on
\begin{align*}
\Delta \cap \{ x \in M_{\mathbb{R}} \mid \langle x,n_{\Gamma} \rangle = b_{\Gamma} + 1  \}
\end{align*}
Then 
\begin{align*}
\ker(\kappa_f^0) = \langle x_1 \cdot \frac{\partial f}{\partial x_1},...,\frac{\partial f}{\partial x_n} \rangle
\end{align*}
More difficult cases might get rather complicated and less unifyable. 
\end{comment}

\leavevmode \\
What is known in the direction of our results: For $Y_f \subset \mathbb{P}^n$ a smooth degree $d$ hypersurface by (\cite[Ch.6]{Voi03})
\[ \ker(\kappa_{f}) \cong  \ker(\kappa_{\mathbb{P},f}) \cong J_{f,griff}^d,  \]
where $J_{f,griff}^d$ denotes the $d$-th homogeneous component of 
\[ J_{f,griff}:= \Big( \frac{\partial f}{\partial x_0},....,\frac{\partial f}{\partial x_n} \Big) \unlhd \mathbb{C}[x_0,...,x_n].  \]
In Example \ref{example_proj_hyp_Griffiths_jac_ideal_ker_Kod_Spenc} we show that if $d \geq n+1$ then our results specialize to this assertion. There are generalizations of this results to quasismooth hypersurfaces in weighted projective spaces (\cite[Thm.4.3.2]{Dol82}). In the dissertation (\cite{Koe91}) it is dealt with families of nondegenerate curves in toric surfaces. In (\cite{Mav03}) both the kernel and the cokernel of the Kodaira-Spencer map is dealt with for \textit{anticanonical hypersurfaces} and in (\cite{Oka87}) the Kodaira-Spencer map has been studied for $\Delta$ an $n$-simplex.
\\ \\
The author declares no conflicts of interest. Data sharing is not applicable to this article as no new data were created or analyzed in this study. The data that support the findings of this study are openly available at arXiv:2203.01092 [math.AG]

	\section{Notation and Background} \label{section_Not_and_background}
	Let $M$ and $N$ be dual lattices and $T = N \otimes_{\mathbb{Z}} \mathbb{C}^* \cong (\mathbb{C}^*)^n$ the $n$-dimensional torus. We fix an $n$-dimensional lattice polytope $\Delta \subset M_{\mathbb{R}}$ and denote the normal fan of a lattice (or more generally rational) polytope $F$ by $\Sigma_F$. We denote the projective toric variety defined via $\Sigma_F$ by $\mathbb{P}_F$. $\Sigma[1]$ denote the rays of the fan $\Sigma$ and $D_i$ (or $D_{\nu}$) the toric divisor to the ray $\nu_i$ (or $\nu$). \\ \\
	Throughout this article we assume $f$ to be a Laurent polynomial with Newton polytope $\Delta$, written
	\[ f = \sum\limits_{m \in M \cap \Delta} a_m x^m.  \]
	Set $\quad  Z_f := \{f=0\} \subset T$ and given another $n$-dimensional rational polytope $F$ let $Z_{F,f}$ or just $Z_{F}$ be the closure of $Z_f$ in $\mathbb{P}_{F}$. We repeat some results and notions introduced in (\cite{Bat22}):
	\begin{remark}  \label{remark_lin_equiv_class_hypersurface}
		\normalfont
		The linear equivalence class of $Z_{F,f}$ just depends on $\Delta$. More precisely let
		\[ \Min_{\Delta}(\nu) := \min\limits_{m \in \Delta} \langle m,\nu \rangle  \]
		for $\nu \in N \setminus \{0\}$. Then
		\[ Z_{F,f} \sim_{lin} - \sum\limits_{\nu \in \Sigma_F[1]} \Min_{\Delta}(\nu) \cdot D_{\nu} \]
		(\cite[Prop.7.1]{Bat22}). Therefore we usually omit $f$ from the notation.
	\end{remark}

	The nondegeneracy of $f$, written $f \in U_{reg}(\Delta)$, means that $Z_f$ is smooth and for every $k$-dimensional face $\Gamma$ of $\Delta$
	\begin{align} \label{formula_alt_condition_derivative_nondegeneracy}
		f_{\vert{\Gamma}}, \, \frac{\partial f_{\vert{\Gamma}}}{\partial x_1},...,  \frac{\partial f_{\vert{\Gamma}}}{\partial x_n},
	\end{align} 
	where $f_{\vert{\Gamma}} := \sum\limits_{m \in M \cap \Gamma} a_m x^m$, have no common zero in the torus orbit $(\mathbb{C}^*)^k \subset \mathbb{P}_{\Delta}$ corresponding to $\Gamma$. Write
	\begin{align} \label{Delta_intersection_of_hyperplanes_presentation}
	\Delta = \{ x \in M_{\mathbb{R}} \mid \langle x, \nu \rangle \geq \Min_{\Delta}(\nu) \quad \forall \, \nu \in N \setminus \{0\}  \}. 
	\end{align}
	we consider the \textit{Fine interior}
	\[ F(\Delta) := \{ x \in M_{\mathbb{R}} \mid \langle x, \nu \rangle \geq \Min_{\Delta}(\nu) + 1 \quad \forall \, \nu \in N \setminus \{0\}  \}.  \]
	In general $F(\Delta)$ is a rational polytope contained in $\Delta$ and contains all interior lattice points of $\Delta$. It is constructed by moving every hyperplane touching $\Delta$ one step into the interior of $\Delta$. Let
	\[ S_{F}(\Delta) := \{ \nu \in N \setminus \{0\} \mid \Min_{F(\Delta)}(\nu) = \Min_{\Delta}(\nu) + 1  \}  \]
	denote the \textit{support} of $F(\Delta)$ to $\Delta$ (those hyperplanes that touch $\Delta$ and moved by one touch $F(\Delta)$).
%	\begin{proposition} \label{proposition_support_vectors_contained_in_convex_hull} (\cite[Prop.3.11]{Bat22}) \\
%		The support vectors $S_{F}(\Delta)$ are contained in the convex span of the rays $\Sigma_{\Delta}[1]$.
%	\end{proposition}
	
	\begin{figure}[H]
		\begin{center}
			
			\begin{tikzpicture}[scale=0.65]
				
				\fill (0,1) circle (2pt);
				\fill (0,2) circle (2pt);	
				\fill (0,3) circle (2pt);
				\fill (1,1) circle (2pt);
				\fill (1,2) circle (2pt);	
				\fill (1,3) circle (2pt);	
				\fill (2,1) circle (2pt);
				\fill (2,2) circle (2pt);	
				\fill (2,3) circle (2pt);	
				\fill (3,1) circle (2pt);
				\fill (3,2) circle (2pt);	
				\fill (3,3) circle (2pt);
				\fill (4,1) circle (2pt);
				\fill (4,2) circle (2pt);	
				\fill (4,3) circle (2pt);

				\draw (0,1) -- (0,3) -- (4,1) -- (0,1);
				
				\draw[dashed] (-1,3) -- (4,3);	
				\draw[arrows=->](-0.5,3)--(-0.5,2);
				\draw[arrows=->](1.7,3)--(1.7,2);
				\draw[arrows=->](3.5,3)--(3.5,2);
				\draw[dashed] (-1,2) -- (4,2);

				\begin{scope}[xshift = 7cm]
					
					\fill (0,1) circle (2pt);
					\fill (0,2) circle (2pt);	
					\fill (0,3) circle (2pt);
					\fill (0,4) circle (2pt);	
					\fill (1,1) circle (2pt);
					\fill (1,2) circle (2pt);	
					\fill (1,3) circle (2pt);
					\fill (1,4) circle (2pt);	
					\fill (2,1) circle (2pt);
					\fill (2,2) circle (2pt);	
					\fill (2,3) circle (2pt);
					\fill (2,4) circle (2pt);		
					
					\draw[->] (1,3) -- (0.025,1.05);
					\draw[->] (1,3) -- (1,3.95);
					\draw[->] (1,3) -- (1.95,3);
					
					\draw[dashed] (0,1) -- (2,3) -- (1,4) -- (0,1);
					\draw[->,dashed,thick] (1,3) -- (1,2.05);
					
				\end{scope}

			\end{tikzpicture}
			\caption{Figure: On the left a polygon $\Delta$, on the right the rays of $\Sigma_{\Delta}$. $F(\Delta)$ equals the interior lattice point of $\Delta$. The illustrations shows that $(0,-1) \in S_F(\Delta)$. In general the support vectors $S_F(\Delta)$ are contained in the convex span of the rays $\Sigma_{\Delta}[1]$ (see \cite[Prop.3.11]{Bat22})} \label{figure_support_of_Fine_interior}
		\end{center}
	\end{figure}
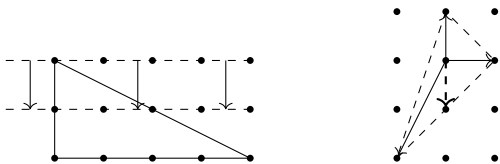

%\section{The canonical closure $C(\Delta)$} \label{section_can_closure}

\begin{definition} \label{definition_can_closure} (\cite[Def.3.13]{Bat22}) \\
	Let $\Delta$ be a lattice polytope with presentation as in (\ref{Delta_intersection_of_hyperplanes_presentation}) and with $F(\Delta) \neq \emptyset$. The polytope
	\[ C(\Delta) := \{ x \in M_{\mathbb{R}}| \, \langle x,\nu \rangle \geq \Min_{\Delta}(\nu) \quad \forall \, \nu \in S_F(\Delta)  \}  \]
	is called the \textit{canonical closure} of $\Delta$.
\end{definition}

\begin{remark} \label{remark_properties_of_CDelta_FDelta_Support}
	\normalfont
	It is clear from the definition that $C(\Delta)$ is a rational polytope and contains $\Delta$. The operators of taking the Fine interior $F$, the canonical closure $C$ and the support $S_F$ could be defined analogously for rational polytopes. 
	%The key point for introducing $C(\Delta)$ is that $\Sigma$ could be chosen as a refinement of $\Sigma_{C(\Delta)}$. Thus there is a birational morphism $\mathbb{P} \rightarrow \mathbb{P}_{C(\Delta)}$. 
\end{remark}

\begin{proof}
	The first formula follows from the exact sequence (\ref{normal_sheaf_defining_exact_sequence}), the vanishing $H^1(\mathbb{P}, \mathcal{O}_{\mathbb{P}}) = 0$ due to Demazure and formula (\ref{formula_polytope_assoc_to_hyp_Y_CDelta}). The second follows from Remark \ref{remark_normal_sheaf_2}.
\end{proof}

	\begin{theorem} (\cite[Thm.8.2]{Bat22}) \\
		Assume $F(\Delta) \neq \emptyset$. Then there is a complete simplicial fan $\Sigma$ with $\Sigma[1] = S_F(\Delta)$ and with associated toric variety $\mathbb{P}$ such that with $Y=Y_f$ the closure of $Z_f$ in $\mathbb{P}$
		\begin{itemize}
			\item $\mathbb{P}$ has $\mathbb{Q}$-factorial terminal singularities 
			\item the adjoint divisor $K_{\mathbb{P}} + Y$ is nef.
		\end{itemize}
	\end{theorem}

	\begin{corollar} \label{corollary_min_model_Z_f} (\cite[Thm.8.2]{Bat22}) \\
		$Y \subset \mathbb{P}$ has at most terminal singularities and $K_{Y}$ is nef. We say that $Y=Y_f$ is a minimal model of $Z_f$.
	\end{corollar}
	
%	By (\cite[Thm.9.2]{Bat22}) the Kodaira dimension of $Y$ is given by
%	\begin{align} \label{formula_Kodaira_dim}
%	\kappa(Y) = \min(n-1, \dim \, F(\Delta)).	
%\end{align} 

\begin{proposition} \label{proposition_hyp_nef_and_big} (\cite[Prop.7.4]{Bat22}) \\
	$Y \subset \mathbb{P}$ defines a nef and big $\mathbb{Q}$-Cartier divisor.
\end{proposition}

\begin{lemma} \label{lemma_torus_orbit_intersection_hyp}
	Let $F$ be a torus orbit of $\mathbb{P}$. If 
	\[ \emptyset \neq F \cap Y \subsetneq F  \]
	then $Y$ intersects $F$ transversely in a subset of codimension $1$.
\end{lemma}

\begin{proof}
	Choose a common refinement $\tilde{\Sigma}$ of $\Sigma$ and $\Sigma_{\Delta}$ in $N$ such that the toric variety $\tilde{\mathbb{P}}$ to $\tilde{\Sigma}$ is smooth. Let $\tilde{Y}$ denote the closure of $Z_f$ in $\tilde{\mathbb{P}}$. Then for $\tilde{F} \subset \tilde{\mathbb{P}}$ a torus orbit either $\tilde{Y}$ is disjoint from $\tilde{F}$ (if $\tilde{F}$ contracts to a torus fixed point on $\mathbb{P}_{\Delta}$) or intersects $\tilde{F}$ transversely in a subset of codimension $1$ (\cite[Prop.3.2.1]{Bat94}). \\
	There is a birational toric morphism $\tilde{\mathbb{P}} \rightarrow \mathbb{P}$ inducing a birational morphism $\tilde{Y} \rightarrow Y$. Assume that $\tilde{F}$ contracts onto $F$. If $\tilde{Y} \cap \tilde{F} \neq \emptyset$ then $Y \cap F \neq \emptyset$ and either $Y$ contains $F$ or intersects $F$ transversely in a subset of codimension $1$.
\end{proof}
	
\begin{remark} \label{remark_V_and_U_open_subsets_smoothness_codimension}
	\normalfont
	The singular loci both of $\mathbb{P}$ and $Y$ are of codimension $\geq 3$ since they are terminal. It will be important to choose smooth open subset $V \subset \mathbb{P}$ and $U \subset Y$ such that $U = V \cap Y$ and $V$ is toric. This works if we define $V$ as the union of all torus orbits of dimension $\geq n-2$.
\end{remark}
	
\begin{remark} \label{remark_complete_lin_system_min_model_can_closure}
	\normalfont
	To a toric divisor 
	\[ D = \sum\limits_{i=1}^r a_i D_i \subset \mathbb{P}  \]
	we associate a polytope
	\[ P_D := \{ x \in M_{\mathbb{R}} \mid \langle x, \nu_i \rangle \geq -a_i \quad \forall n_i \in \Sigma[1] \}  \]
	counting the global sections of $D$, that is $H^0(\mathbb{P}, \mathcal{O}_{\mathbb{P}}(D)) \cong L(P_D)$ (compare \cite[Prop.4.3.2]{CLS11}).	By the formula in Remark \ref{remark_lin_equiv_class_hypersurface} and by (\cite[Prop.4.3]{Bat22}) to $Y$ is associated the polytope $C(\Delta)$
	%	\[ \{ x \in M_{\mathbb{R}}| \, \langle x,n_i \rangle \geq \Min_{\Delta}(n_i) \quad \forall n_i \in S_F(\Delta)  \}  \]
	%	By (\cite[Prop.4.3]{Bat22}) $\Sigma_{C(\Delta)}[1] \subset S_F(\Delta)$ and
	%	\[ \Min_{\Delta}(n_i) = \Min_{C(\Delta)}(n_i) \quad \forall n_i \in S_F(\Delta).  \]
	%	Thus to $Y$ is associated $C(\Delta)$, 
	that is
	\begin{align} \label{formula_polytope_assoc_to_hyp_Y_CDelta}
		H^0(\mathbb{P}, \mathcal{O}_{\mathbb{P}}(Y)) \cong L(C(\Delta)). 
	\end{align}
	It follows in the same way that to $Y+K_{\mathbb{P}}$ is associated $F(\Delta)$.
\end{remark}

\begin{corollar} \label{proposition_criterion_Y_Cartier_divisor}\textbf{}
	$Y \subset \mathbb{P}$ is Cartier if and only if $C(\Delta)$ is a lattice polytope.
\end{corollar}

\begin{proof}
	By Remark \ref{remark_lin_equiv_class_hypersurface}
	\begin{align*}
		Y \sim_{lin}  - \sum\limits_{\nu \in \Sigma[1]} \Min_{\Delta}(\nu) D_{\nu} =  - \sum\limits_{\nu \in \Sigma[1]} \Min_{C(\Delta)}(\nu) D_{\nu}.
	\end{align*}
	The function $\Min_{C(\Delta)}: N_{\mathbb{R}} \rightarrow \mathbb{R}$ has the advantage that it is linear on the cones of $\Sigma$ since $\Sigma$ refines the normal fan of $C(\Delta)$. Thus $Y \subset \mathbb{P}$ is Cartier if and only if $\Min_{C(\Delta)}$ is a \textit{support function} for $Y$, that is 
	\[ \Min_{C(\Delta)}(N) \subset \mathbb{Z}.  \]
	Obviously this is the case if and only if $C(\Delta)$ is a lattice polytope.
\end{proof}

\begin{remark} 
	\normalfont
	In particular if $C(\Delta) = \Delta$ then $Y$ will be Cartier. In (\cite[Ex.3.1.5]{Bat22}) there is mentioned an example of a $5$-dimensional lattice polytope $\Delta$ for which $C(\Delta)$ is not a lattice polytope.
\end{remark}

\section{Reflexive sheaves and MCM-sheaves} \label{section_refl_torsion_free_sheaves}
%\section{Tangent sheaf and sheaves of differential p-forms} \label{section_tangent_shaf_p_diff_forms}

Throughout this section $X$ denotes an $n$-dimensional normal irreducible variety with $j: U \rightarrow X$ the inclusion of the smooth locus of $X$.

	\begin{definition}
		A coherent sheaf $\mathcal{F}$ on $X$ is called reflexive if the natural map $\mathcal{F} \rightarrow \mathcal{F}^{**}$ is an isomorphism, where $\mathcal{F}^{**}$ denotes the double dual of the sheaf $\mathcal{F}$. $\mathcal{F}^{**}$ is called the reflexive hull of $\mathcal{F}$.
		
	\end{definition}
	
	\begin{remark} \label{remark_coh_sheaf_reflexive} \,
		(\cite[Prop.8.0.1, Thm.8.0.4]{CLS11}) \\
		\normalfont
		Given an open subset $j:U \subset X$ with $\codim_X(X \setminus U) \geq 2$ a reflexive sheaf is uniquely determined by its restriction to $U$, that is
		\begin{align} \label{example_sheaf_reflexive}
			\mathcal{F} \cong j_*(\mathcal{F}_{\vert{U}}).
		\end{align}
		Furthermore if $\mathcal{F}$ is a coherent sheaf with $\mathcal{F}_{\vert{U}}$ locally free and $\codim_X(X \setminus U) \geq 2$ then $j_{*}(\mathcal{F}_{\vert{U}})$ is reflexive (\cite[Prop.2.12]{Sch08}). The dual of a coherent sheaf on a normal variety is always reflexive. In particular the reflexive hull of a coherent sheaf is reflexive.
	\end{remark}

The tensor product of two reflexive is defined by
\[ \mathcal{F} \otimes_r \mathcal{G}:= (\mathcal{F} \otimes \mathcal{G})^{**}. \]
Given two Weil divisors $D,D'$ on $X$ we have $\mathcal{O}(D+D') \cong \mathcal{O}(D) \otimes_r \mathcal{O}(D')$.

\begin{comment}	
	\begin{remark} \label{remark_tensor_product_refl-sheaf_line_bundle}
		\normalfont
		The tensor product $\mathcal{F} \otimes \mathcal{G}$ of two reflexive sheaves need not be reflexive: For example for $X:= \{(x,y) \in \mathbb{A}^2 \mid x \cdot y = 0  \}$ and $D_1 = \{x=0\}$, $D_2 = \{y=0\}$ 
		\[ \mathcal{O}_X(D_1) \otimes \mathcal{O}_X(D_2)  \]
		is not reflexive. We may solve this problem by defining the reflexive tensor product of two sheaves as
		\[ \mathcal{F} \otimes_r \mathcal{G}:= (\mathcal{F} \otimes \mathcal{G})^{**}. \]
		In this sense given two Weil divisors $D, \, D'$ the rank one reflexive sheaf $\mathcal{O}(D+D')$ is isomorphic to
		\[ \mathcal{O}(D) \otimes_r \mathcal{O}(D').  \]
		If $\mathcal{F}$ is reflexive and $\mathcal{L}$ a line bundle then $\mathcal{F} \otimes \mathcal{L}$ remains reflexive. This may be checked stalkwise using (\cite[Ch.3, Prop.6.8]{Hart77}).
	\end{remark}
\end{comment}

%	\section{Tangent sheaf and sheaves of differential p-forms} \label{section_tangent_shaf_p_diff_forms}

	\begin{definition}
		Let
		\begin{align*}
			& \Omega_X^p := j_* \Omega_{U}^p \quad 1 \leq p \leq n, \\
			& T_X := (\Omega_X^1)^*,
		\end{align*} 
		be the sheaf of differential $p$-forms and the tangent sheaf. The sheaf
		\[ N_{Y/\mathbb{P}} := \iota_* \mathcal{O}_U(Y_{\vert{U}})  \]
		is called the normal sheaf of $Y$ in $\mathbb{P}$.
	\end{definition}

	\begin{remark} \label{remark_Kaehler_diff_forms_Aut_group}
		\normalfont
		There is a different method for the construction of $T_{X}$: Let $\Omega_{X,\textrm{K\"ahl}}^p$ denote the sheaf of K\"ahler $p$-differentials on $X$ and 
		\[ T_{X,\textrm{K\"ahl}}:= (\Omega_{X,\textrm{K\"ahl}}^1)^{*}.  \]
		Then $\Omega_{X,\textrm{K\"ahl}}^1$ is coherent and thus by Remark \ref{remark_coh_sheaf_reflexive} its dual $T_{X,\textrm{K\"ahl}}$ is reflexive and coincides with $T_{X}$ on the smooth locus $U$ of $X$. Since both sheaves are reflexive
		\[ T_{X,\textrm{K\"ahl}} \cong T_{X}. \]

		By (\cite[Ch.\RNum{6}.1]{MuOd15})
		\begin{align} \label{formula_Kaehler_tangential_Lie_algebra}
		H^{0}(X, T_{X,\textrm{K\"ahl}}) \cong \Lie(\Aut(X)), 
		\end{align}
		where $\Lie(\Aut(X))$ denotes the Lie algebra of the automorphism group of $X$. 		
\end{remark}
		
		\begin{comment}
		In particular if $\Aut(X)$ is a finite group then $H^0(X,T_X) = 0$. For $V$ a normal projective toric variety $\Aut(V)$ is an algebraic group (\cite[Prop.4.3]{Cox95}), whereas for $Y$ a normal projective variety it might happen that $H^0(Y,T_Y) = 0$ but $\Aut(Y)$ has infinite order.
	
	\begin{example}
		\normalfont
		Take the Fermat quartic $Y \subset \mathbb{P}^3$
		\[ 0 = x_0^4 + ... + x_3^4,  \]
		which defines a K3-surface. Then $h^0(Y,T_Y)=0$ but $Y$ has infinite automorphism group (see \cite[Thm.5]{ShIn10}).
	\end{example}
	
\end{comment}

%	\section{The Normal sheaf and MCM-sheaves}

%Equivalently we could have first defined the conormal sheaf $N_{Y/\mathbb{P}}^*$ as $I_Y/I_Y^2$, where $I_Y$ denotes the ideal sheaf of $Y$ in $\mathbb{P}$ and then $N_{Y/\mathbb{P}}$ as
%\[ N_{Y/\mathbb{P}} := (N_{Y/\mathbb{P}}^*)^*.  \]

%We recall the following definition
In order to apply Serre duality we need sheaves with the following property
\begin{definition} (\cite[Def. before Thm.9.2.12]{CLS11}, \cite[Def.5.2]{Kol13}) \\
	Let $\mathcal{F}$ be a coherent sheaf on $X$. Then $\mathcal{F}$ is called maximal Cohen-Macaulay (for short:MCM) if for all $x \in X$ the stalk $\mathcal{F}_x$ is a Cohen-Macaulay module over $\mathcal{O}_{X,x}$ of dimension $n$.
\end{definition}

\begin{proposition} \label{proposition_CM_sheaves} \cite[Proof of Thm.9.2.10]{CLS11} \\
	Let $V$ be an $n$-dimensional complete toric variety to a simplicial fan and $D$ a $\mathbb{Q}$-Cartier divisor on $V$. Then the sheaves $\Omega_V^p$ for $p=0,...,n$ and $\mathcal{O}_V(D)$ are MCM.
\end{proposition}

Let $\mathcal{F}$ be a MCM-sheaf on $X$. Then Serre duality (\cite[Thm.5.71, Prop.5.75]{KM98}) gives
\begin{align*}
	H^k(X, \mathcal{F}) \cong H^{n-k}(X, \Hom(\mathcal{F}, \mathcal{O}_X(K_X)))^*.
\end{align*}

For $Z \subset X$ a closed subset, $\mathcal{F}$ a coherent sheaf on $X$, we work with the \textit{local cohomology groups} $H_Z^k(X,\mathcal{F})$ on $X$. This allows us to extend exact sequences 
%like the one in Proposition \ref{proposition_exact_normal_sheaf_sequence_defining} below
which exist on the smooth locus $U$ of $X$ to all of $X$. First we recall a vanishing Theorem:

\begin{theorem} \label{theorem_Grothendieck_van_local_coh} (\cite[Prop.3.3]{HaKo04}) \\
	Let $X$ be an $n$-dimensional algebraic variety and $Z \subset X$ a closed subset with $\codim_X(Z) \geq r$. Let $\mathcal{F}$ be a coherent sheaf on $X$ which is MCM. Then
	\begin{align*}
		H_Z^k(X, \mathcal{F}) = 0  \quad k=0,...,r-1.
	\end{align*}
\end{theorem}

	\begin{proposition} \label{proposition_exact_normal_sheaf_sequence_defining}
		The sequence of sheaves
		\begin{align} \label{normal_sheaf_defining_exact_sequence}
			0 \rightarrow \mathcal{O}_{\mathbb{P}} \rightarrow \mathcal{O}_{\mathbb{P}}(Y) \rightarrow N_{Y/\mathbb{P}} \rightarrow 0
		\end{align}
	is exact.
	\end{proposition}

	\begin{proof}
		If $Y \subset \mathbb{P}$ is Cartier (see Corollary \ref{proposition_criterion_Y_Cartier_divisor} for a criterion) then there is nothing to show, but in general $Y$ just defines a $\mathbb{Q}$-Cartier divisor. \\ We may assume that $n = \dim \, \mathbb{P} \geq 3$. Let $\iota: V \rightarrow \mathbb{P}$ be the inclusion of the union of all torus orbits of dimension $\geq n-2$ and $U:= V \cap Y$. If we take the pushforward under $\iota_*$ of an ideal sheaf sequence for $U \subset V$ we get
		\[0  \rightarrow \mathcal{O}_{\mathbb{P}} \rightarrow \mathcal{O}_{\mathbb{P}}(Y) \rightarrow N_{Y/\mathbb{P}} \rightarrow R^1 \iota_*(\mathcal{O}_V).  \]
		 $R^1 \iota_*(\mathcal{O}_{V})$ is the sheaf associated to the presheaf
		\[ W \mapsto H^1(W \cap V,  \mathcal{O}_{W \cap V}). \]
		Assume that $W \subset \mathbb{P}$ is affine, $V \subset W$ an open subset such that $\codim_W(Z) \geq 3$ for $Z:= W \setminus V$. We have to show that $H^1(V, \mathcal{O}_{V}) = 0$. There is a local cohomology exact sequence (\cite[Cor.1.9]{Gro67})
		\[ ... \rightarrow H^1(W,\mathcal{O}_{W}) \rightarrow H^1(V, \mathcal{O}_{V}) \rightarrow H_Z^2(W, \mathcal{O}_{W}) \]
		Since $W$ is affine $H^1(W,\mathcal{O}_W) = 0$ by Serre's criterion and $H_Z^2(W, \mathcal{O}_{W}) = 0$ by Theorem \ref{theorem_Grothendieck_van_local_coh} since $\mathcal{O}_{\mathbb{P}}$ is MCM and $\codim_W(Z) \geq 3$. 
	\end{proof}

We will use this exact sequence and the vanishing $H^1(\mathbb{P}, \mathcal{O}_{\mathbb{P}}) = 0$ in Remark \ref{remark_complete_lin_system_min_model_can_closure} to compute $H^0(Y,N_{Y/\mathbb{P}})$ by identifying $Y \subset \mathbb{P}$ with a torus invariant divisor.

\begin{corollar} \label{lemma_normal_sheaf_Y_P_CM}
	The normal sheaf $N_{Y/\mathbb{P}}$ is MCM.
\end{corollar}

\begin{proof}
	Consider the exact sequence (\ref{normal_sheaf_defining_exact_sequence}). For $y \in Y$ the $\mathcal{O}_{Y,y}$-module $(N_{Y/\mathbb{P}})_y$ has dimension $n-1$ and by (\cite[Cor.2.62(3)]{Kol13}) it is Cohen-Macaulay since both $\mathcal{O}_{\mathbb{P}}$ and $\mathcal{O}_{\mathbb{P}}(Y)$ are MCM sheaves.
\end{proof}

\begin{remark} \label{remark_normal_sheaf_2}
	\normalfont
	The normal sheaf $N_{Y_f/\mathcal{X}}:= (I_{Y_f} \otimes_{\mathcal{O}_{\mathcal{X}}} \mathcal{O}_{Y_f})^*$ is trivial of rank $l(\Delta)-1$, where $\mathcal{I}_{Y_f}$ denotes the ideal sheaf of $Y_f \subset \mathcal{X}$. 
	\begin{comment}
	Given $f \in B$ choose $l(\Delta)-1$ hyperplanes $H_1,...,H_{l(\Delta)-1}$ in $B$ which cut out $\{f\}$. Then $Y_f$ is the complete intersection in $\mathcal{X}$ defined by
	\[ G_1:= pr_2^*(H_1),...,G_{l(\Delta)-1}:= pr_2^*(H_{l(\Delta)-1}).  \]
	%We derive from the exact sequence
	%\[ 0 \rightarrow \mathcal{O}_{\mathcal{X}}(-\sum\limits_{i=1}^{l(\Delta)-1} G_i) \rightarrow \bigoplus\limits_{i=1}^{l(\Delta)-1} \mathcal{O}_{\mathcal{X}}(-G_i) \rightarrow I_{Y_f} \rightarrow 0  \]
	%that
	Thus
	\[ N_{Y_f/\mathcal{X}}  \cong \bigoplus\limits_{i = 1}^{l(\Delta)-1} \mathcal{O}_{Y_f}(G_i). \]
	We may switch to (linear equivalent) hyperplanes $H_1',...,H_{l(\Delta)-1}'$ with \\
	$\{f\} \notin H_1' \cap ... \cap H_{l(\Delta)-1}'$.	Then 
	\[ \mathcal{O}_{Y_f}(G_i) \cong  \mathcal{O}_{Y_f}(G_i') \cong \mathcal{O}_{Y_f} \] 
	for $G_i':= pr_2^*(H_i')$. Thus
	\begin{align*}
		N_{Y_f/\mathcal{X}} \cong \bigoplus\limits_{i = 1}^{l(\Delta)-1} \mathcal{O}_{Y_f} \cong T_{B,f} \otimes_{\mathbb{C}} \mathcal{O}_{Y_f}.
	\end{align*}
\end{comment}
\end{remark}

\begin{corollar} \label{remark_normal_sheaf}
	\begin{align*}
		& H^0(Y,N_{Y/\mathbb{P}}) \cong L(C(\Delta))/ \mathbb{C} \cdot f, \\
		& H^0(Y,N_{Y/\mathcal{X}}) \cong L(\Delta) / \mathbb{C} \cdot f. 
	\end{align*}
\end{corollar}

\section{Kodaira-Spencer maps} \label{section_Kod_Spencer_maps}

	Let $B$ be the image of $U_{reg}(\Delta)$ in $\mathbb{P}L(\Delta)$. We consider the following natural family
		\[ \mathcal{X} := \{ (x,f) \in \mathbb{P} \times B| \, x \in Y_{f} \} \overset{pr_2}{\rightarrow} B.  \]
	
	\begin{definition} \normalfont
	Given a full-dimensional rational polytope $P \subset M_{\mathbb{R}}$ let $L(P)$ be the $\mathbb{C}$-vector space with basis $x^m$, $m \in \Delta \cap M$ and $l(P) = \# (P \cap M)$ its dimension. 
	\end{definition}
	
	Let $D:= \Spec \, \mathbb{C}[\epsilon]/(\epsilon^2)$ denote the dual numbers. Given a tangent vector $\phi: D \rightarrow B$ of $B$ at $f$ the pullback 
	\[ \mathcal{X} \times_B D \rightarrow D  \]
	defines an \textit{infinitesimal deformation} (of first order) of $Y_f$ in $\mathcal{X}$. Passing from $\mathcal{X} \rightarrow B$ to $\mathcal{X} \times_{\phi} D \rightarrow D$ some information gets lost. Still we can extract interesting information out of the situation (compare Kodaira \cite[Thm.4.3]{Kod86}). \\
	There are isomorphisms
	\begin{align*}
	& H^0(Y, N_{Y/\mathcal{X}}) \quad \, \, \, \cong \{ \textrm{Inf. def. of } Y_f \textrm{ in } \mathcal{X} \}/\textrm{iso.} \\
	& H^0(Y, N_{Y/\mathbb{P}}) \quad \, \, \, \, \cong \{ \textrm{Inf. def. of } Y_f \textrm{ in } \mathbb{P} \} / \textrm{iso.} \, \\
	& \Ext_{\mathcal{O}_{Y}}^1(\Omega_{Y}^1, \mathcal{O}_{Y}) \cong \{ \textrm{Inf. def. of } Y_f \} / \textrm{iso.}
	\end{align*}
\begin{comment}
	and \textit{Kodaira-Spencer} maps
	\begin{align*}
	& \kappa_f: H^0(Y,N_{Y/\mathcal{X}}) \rightarrow \Ext_{\mathcal{O}_{Y}}^1(\Omega_{Y}^1, \mathcal{O}_{Y}) \\
	& \kappa_{\mathbb{P},f}: H^0(Y, N_{Y/\mathbb{P}}) \rightarrow \Ext_{\mathcal{O}_{Y}}^1(\Omega_{Y}^1, \mathcal{O}_{Y})
	\end{align*}
	associating to such a deformation its equivalence class in $\Ext_{\mathcal{O}_{Y_f}}^1(\Omega_{Y_f}^1, \mathcal{O}_{Y_f})$. Both $\kappa_{f}$ and $\kappa_{\mathbb{P},f}$ may be realized as coboundary maps by restricting to the smooth locus $U$ of $Y_f$: 
\end{comment}

	Let us derive two \textit{Kodaira-Spencer} maps in this context: There is an isomorphism
	\[ H^1(U,T_{U}) \cong \Ext_{\mathcal{O}_{U}}^1(\Omega_{U}^1, \mathcal{O}_{U}). \]
	since $U$ is smooth. Besides by (\cite[Lemma (12.5.6)]{KM92})
	\[ \Ext_{\mathcal{O}_Y}^1(\Omega_{Y}^1, \mathcal{O}_Y) \cong \Ext_{\mathcal{O}_U}^1(\Omega_U^1, \mathcal{O}_U) \cong H^1(U,T_U). \]
	since $\codim_Y(Y \setminus U) \geq 3$. For $V \subset \mathbb{P}$ an open subset as in Remark (\ref{remark_V_and_U_open_subsets_smoothness_codimension}) with $V \cap Y = U$ the Kodaira-Spencer map
	\[  H^0(U,N_{U/V}) \rightarrow H^1(U,T_U)  \]
	is the coboundary map in the tangent sheaf sequence for $U \subset V$. We derive one Kodaira-Spencer map $\kappa_{\mathbb{P},f}$ from the commutative diagram
	\[  
	\begin{tikzcd}[scale cd = 0.9]
		0 \arrow[r] & H^0(U,T_U) \arrow[r] \arrow{d}{\cong} & H^{0}(U,T_{V \vert{U}}) \arrow[r] \arrow{d}{\cong} & H^{0}(U, N_{U/V}) \arrow{d}{\cong} \arrow{r} & H^1(U,T_U)  \arrow{d}{\cong} \\
		0 \arrow[r] & H^0(Y,T_Y) \arrow[r] & H^{0}(Y, T_{\mathbb{P} \vert{Y}}^{**}) \arrow[r] & H^{0}(Y, N_{Y/\mathbb{P}}) \arrow[dashed]{r}{\kappa_{\mathbb{P}}=\kappa_{\mathbb{P},f}} & \Ext_{\mathcal{O}_Y}^1(\Omega_{Y}^1, \mathcal{O}_Y)
	\end{tikzcd}
	\]
	The other Kodaira-Spencer map $\kappa_f:H^0(Y,N_{Y/\mathcal{X}}) \rightarrow \Ext_{\mathcal{O}_{Y}}^1(\Omega_{Y}^1, \mathcal{O}_{Y})$ is gotten similarly by working with $\mathcal{X}$ instead of $\mathbb{P}$. These two maps associate two a deformation in $\mathbb{P}$ ($\mathcal{X}$) its equivalence class.

\begin{comment}
	\begin{definition} \label{def_number_of_moduli}
	We call $\dim \, Im(\kappa_f)$ the number of moduli of $\mathcal{X} \rightarrow B$.
	\end{definition}
\end{comment}

%	\section{The Kodaira-Spencer map} \label{section_Kod_Spencer_maps}

\section{Mavlyutov's vanishing Theorem} \label{section_Mavly_van_result}

%The following Theorem treats one situation where a vanishing Theorem due to Mavlyutov applies:

\begin{theorem} \label{theorem_van_Mavlyutov} (\cite[Thm.2.4]{Mav08}, \cite[Thm.9.3.3]{CLS11}) \\
	Let $V$ be an $n$-dimensional complete toric variety to a simplicial fan. If D is a nef Cartier divisor on $V$, then
	\[ H^{p}(V, \Omega_{V}^q \otimes \mathcal{O}(D)) = 0  \]
	for $p> q$.
\end{theorem}

\begin{construction} \label{construction_frob_split}
	(Multiplication morphism) \\ (\cite[2.5, Prop.3.2]{Fuj06}, \cite[Lemma 9.2.6, Proof of Thm. 9.3.1]{CLS11}) \\
	\normalfont
	Let $V$ be a normal (not necessarily complete) toric variety. There is a useful construction to reduce computations of cohomology groups of $\mathbb{Q}$-Cartier Weil divisors $D$ to cohomology of multiples $mD$ of $D$, which are Cartier. \\ \\
	Namely for $l \in \mathbb{N}_{\geq 1}$ the map $\overline{\phi}_l: N \rightarrow N$ given by 
	\[  n \mapsto l \cdot n \] 
	yields a toric morphism $\phi_l: V \rightarrow V$. $\phi_l$ induces an injection
	\begin{align*}
		H^p(V, \Omega_{V}^q \otimes_r \mathcal{O}(D)) & \rightarrow H^p(V, \Omega_{V}^q \otimes_r \mathcal{O}(lD))  \\
		& \cong H^p(V, \Omega_{V}^q \otimes \mathcal{O}(lD)), \quad p \geq 0,
	\end{align*} 
	where the last isomorphism follows %from Remark \ref{remark_tensor_product_refl-sheaf_line_bundle} 
	since $\mathcal{O}(lD)$ is Cartier. \\
	For us this result isr powerful in connection with Theorem \ref{theorem_van_Mavlyutov} above.

\end{construction}

\section{Computation of $\ker(\kappa_{\mathbb{P},f})$}
\label{section_number_of_moduli}

	\begin{theorem} \label{theorem_number_of_moduli}
		Let $\Delta$ be an $n$-dimensional lattice polytope, where $n \geq 2$, with $l^*(\Delta) > 0$. If $n=2$ assume that $l^*(\Delta) \geq 2$. Then given $f \in B$
	\begin{align} 
		\ker(\kappa_{\mathbb{P},f}) \cong  \Lie \, \Aut(\mathbb{P}).
		\end{align} 		
	\end{theorem}

	\begin{proof}
		\normalfont	
		We write $Y$ and $\kappa_{\mathbb{P}}$ for $Y_f$ and $\kappa_{\mathbb{P},f}$. By the exact sequence from section \ref{section_Kod_Spencer_maps}
		\[ \ker(\kappa_{\mathbb{P}}) \cong  H^0(Y,T_{\mathbb{P} \vert{Y}}^{**})/H^0(Y,T_Y).  \]
		By Lemma (\ref{theorem_number_of_moduli_2}) below if $n \geq 3$ then
		\[ H^0(Y,T_Y) = 0. \]
		If $n=2$ this follows from $g(Y) \geq 2$ since the Fine interior is not a point. By Remark \ref{remark_Kaehler_diff_forms_Aut_group} we are left to show that 
		\[ H^0(Y,T_{\mathbb{P} \vert{Y}}^{**}) \cong H^0(\mathbb{P}, T_{\mathbb{P}}).  \]
		
		\begin{lemma}
			There is an exact sequence
					\begin{align*}
				0 \rightarrow T_{\mathbb{P}} \otimes_r \mathcal{O}(-Y) \rightarrow T_{\mathbb{P}} \rightarrow T_{\mathbb{P} \vert{Y}}^{**} \rightarrow 0.
			\end{align*}	
		\end{lemma}
		
		\begin{proof}		
		Let $\iota: V \subset \mathbb{P}$ be the inclusion of the union of all torus orbits of dimension $\geq n-2$ and $U = V \cap Y$. Consider the exact sequence
		\[ 0 \rightarrow T_V \otimes \mathcal{O}(-Y_{\vert{V}}) \rightarrow T_V \rightarrow T_{V \vert{U}} \rightarrow 0  \]
		and take the pushforward under $\iota$. Let
		\[ \mathcal{F} :=  \Omega_{\mathbb{P}}^{n-1} \otimes_r \mathcal{O}(-Y-K_{\mathbb{P}}) \cong  T_{\mathbb{P}} \otimes_r \mathcal{O}(-Y).    \]
		Then $R^1 \iota_*(\mathcal{F}_{\vert{V}}) = 0$ since given $x \in V$ choose an affine toric neighborhood $W$ of $x$ and (using Construction \ref{construction_frob_split}) replace $\mathcal{O}_W(-Y-K_{\mathbb{P}})$ by an $m$-times multiple which is Cartier in the definition of $\mathcal{F}$. Then by Proposition \ref{proposition_CM_sheaves} $\mathcal{F}_{\vert{W}}$ is MCM and the vanishing
		\[ H^1(V \cap W, \mathcal{F}_{\vert{V \cap W}}) = 0  \]
		follows as in the proof of Proposition \ref{proposition_exact_normal_sheaf_sequence_defining} from a local cohomology sequence.	
	\end{proof}

		Take the long exact cohomology sequence
		\begin{align*} 0 &\rightarrow H^{0}(\mathbb{P}, T_{\mathbb{P}} \otimes_r \mathcal{O}(-Y)) \rightarrow H^{0}(\mathbb{P},T_{\mathbb{P}}) \rightarrow H^{0}(Y,T_{\mathbb{P} \vert{Y}}^{**}) \\
			& \rightarrow H^{1}(\mathbb{P},T_{\mathbb{P}} \otimes_r \mathcal{O}(-Y)).
		\end{align*}

		Write
		\[ T_{\mathbb{P}} \otimes_r \mathcal{O}(-Y) \cong \Omega_{\mathbb{P}}^{n-1} \otimes_r \mathcal{O}(-Y-K_{\mathbb{P}}).  \]
		By Construction \ref{construction_frob_split} we have
		\[ H^k(\mathbb{P}, \Omega_{\mathbb{P}}^{n-1} \otimes_r \mathcal{O}(-Y-K_{\mathbb{P}})) \subset H^k(\mathbb{P}, \Omega_{\mathbb{P}}^{n-1} \otimes \mathcal{O}(-mY-mK_{\mathbb{P}}))  \]	
		where $m \geq 1$ is such that $mY, mK_{\mathbb{P}}$ are Cartier. Now use Serre duality and Theorem \ref{theorem_van_Mavlyutov} to deduce
		\begin{align*}
			 H^k(\mathbb{P}, \Omega_{\mathbb{P}}^{n-1} \otimes \mathcal{O}(-mY-mK_{\mathbb{P}})) \cong &H^{n-k}(\mathbb{P}, \Omega_{\mathbb{P}}^1 \otimes \mathcal{O}(mY+mK_{\mathbb{P}})) = 0
		\end{align*}
		for $k=0,1$ if $n \geq 3$. If $n=2$ the vanishing for $k=1$ follows from the precise formula in (\cite[Cor.2.7]{Mav08}) (to $Y
		+K_{\mathbb{P}}$ is associated the polytope $F(\Delta)$, see Remark \ref{remark_complete_lin_system_min_model_can_closure}).

\end{proof}

\begin{lemma} \label{theorem_number_of_moduli_2}
	Let $\Delta$ be an $n$-dimensional lattice polytope with $n \geq 3$ and $l^*(\Delta) > 0$. Then
	\[ H^0(Y,T_Y) = 0. \]
\end{lemma}
\normalfont

\begin{proof}
	Using an ideal sheaf sequence we get
	\begin{align*}
		0 \rightarrow  T_Y \rightarrow \Omega_Y^{n-1} \rightarrow ...
	\end{align*}
	where we have used $\Omega_Y^{n-1} \cong T_Y \otimes_r \mathcal{O}(K_Y)$ and $h^0(Y,\mathcal{O}(K_Y)) = l^*(\Delta) > 0$. Thus
	\begin{align*}
		h^0(Y,T_Y) &\leq h^0(Y, \Omega_Y^{n-1}) \\
		& = 0
	\end{align*}
	For the last vanishing: Take a common toric resolution of singularities $\mathbb{P} \overset{p}{\leftarrow} \tilde{\mathbb{P}} \rightarrow \mathbb{P}_{\Delta}$ with closure $\tilde{Y}$ of $Z_f$ in $\tilde{\mathbb{P}}$. Then $h^0(\tilde{Y}, \Omega_{\tilde{Y}}^{n-1}) = 0$ by (\cite{DK86}) and $p_* \Omega_{\tilde{Y}}^{n-1} = \Omega_{Y}^{n-1}$.
\end{proof}

We guess that the weaker assumption $F(\Delta) \neq \emptyset$ is sufficient for the above Lemma.

\section{The cokernel of $\kappa_{\mathbb{P},f}$}

\begin{corollar} \label{corollar_all_inf_def_come_from...}
	Given the conditions of the theorem if $n \geq 4$ then the following sequence is exact
	\begin{align*}
		0 \rightarrow Im(\kappa_{\mathbb{P},f}) \rightarrow \Ext_{\mathcal{O}_{Y}}^1(\Omega_{Y}^1, \mathcal{O}_Y) \rightarrow \Ext_{\mathcal{O}_{\mathbb{P}}}^1(\Omega_{\mathbb{P}}^1, \mathcal{O}_{\mathbb{P}}) \rightarrow 0 
	\end{align*}
	that is all infinitesimal deformations of $Y_f$ arise from $Im(\kappa_{\mathbb{P},f})$ or from infinitesimal deformations of $\mathbb{P}$.
\end{corollar}

\begin{proof}
	Let $V$ denote the smooth locus of $\mathbb{P}$, $Z= \mathbb{P} \setminus V$ and $U= V \cap Y$. Then $Z$ has codimension $\geq 3$ in $\mathbb{P}$ and $Y \setminus U$ has codimension $\geq 2$ in $Y$. 
	We could extend the upper exact sequence in Section \ref{section_Kod_Spencer_maps} to
	\begin{align*}
		0 & \rightarrow H^0(U,T_U) \rightarrow H^0(U, T_{V \vert{U}}) \rightarrow H^0(U,N_{U/V}) \overset{\kappa_{\mathbb{P}}}{\rightarrow} H^1(U,T_U) \\
		& \rightarrow H^1(U,T_{V \vert{U}}) \rightarrow H^1(U,N_{U/V})
	\end{align*}
	 By the exact sequence (\ref{normal_sheaf_defining_exact_sequence}) $H^1(Y,N_{Y/\mathbb{P}}) = 0$. Relate $H^1(U,N_{U/V})$ and $H^1(Y,N_{Y/\mathbb{P}})$ via a local cohomology sequence and use (Lemma \ref{lemma_normal_sheaf_Y_P_CM}, Theorem \ref{theorem_Grothendieck_van_local_coh})
	\[ H_{Z \cap Y}^k(Y,N_{Y/\mathbb{P}}) = 0 \quad k < n-1.  \]
	It follows $H^1(U,N_{U/V}) = 0$. For $T_{V \vert{U}}$ use an ideal sheaf sequence
	\begin{align*}
		H^1(V,T_V \otimes \mathcal{O}(-Y_{\vert{V}})) \rightarrow H^1(V,T_V) \rightarrow H^1(U,T_{V \vert{U}}) \rightarrow H^2(V,T_V \otimes \mathcal{O}(-Y_{\vert{V}})).
	\end{align*}
	$V$ is a toric variety and by Construction \ref{construction_frob_split}
	\begin{align*}
		H^k(V,T_V \otimes \mathcal{O}(-Y_{\vert{V}})) &= H^k(V,\Omega_V^{n-1} \otimes \mathcal{O}(-Y_{\vert{V}} - K_{\mathbb{P} \vert{V}})) \\
		& \subset H^k(V, \Omega_V^{n-1} \otimes \mathcal{O}(-mY_{\vert{V}} - mK_{\mathbb{P} \vert{V}}))
	\end{align*}
	where $mY, mK_{\mathbb{P}}$ are Cartier. We show that these terms vanish for $k=1,2$ and the result follows since
	\[ H^1(V,T_V) \cong \Ext_{\mathcal{O}_{\mathbb{P}}}^1(\Omega_{\mathbb{P}}^1, \mathcal{O}_{\mathbb{P}})  \]
	as in section \ref{section_Kod_Spencer_maps}. Use a local cohomology sequence
	\[ H^k(\mathbb{P}, \mathcal{F}) \rightarrow H^k(V, \mathcal{F}_{\vert{V}}) \rightarrow H_Z^{k+1}(\mathbb{P}, \mathcal{F})  \]
	where
	\[ \mathcal{F}:= \Omega_{\mathbb{P}}^{n-1} \otimes \mathcal{O}(-mY - mK_{\mathbb{P}}). \]
	$\mathcal{F}$ is CM by Lemma \ref{lemma_normal_sheaf_Y_P_CM}. As in the Lemma above
	\[ H^k(\mathbb{P}, \mathcal{F}) \cong H^{n-k}(\mathbb{P}, \Omega_{\mathbb{P}}^1 \otimes \mathcal{O}(mY+mK_{\mathbb{P}})) = 0 \quad k=1,2.   \]
	Further by Theorem \ref{theorem_Grothendieck_van_local_coh}
	\[ H_Z^k(\mathbb{P}, \mathcal{F}) = 0 \quad k<n-1.  \]
\end{proof}

\begin{example}
	\normalfont
	If $\mathbb{P} = \mathbb{P}^3$, $\Delta = 4 \cdot \Delta_3$ then $\dim \, Im(\kappa) = 19$ and $H^1(\mathbb{P},T_{\mathbb{P}}) = 0$. But $h^1(Y,T_Y) = 20$ since $Y$ is a K3 surface. The above Corollary does not apply since $n=3$. %and the reason for this is that there are non-algebraic deformations of K3 surfaces.
\end{example}

\section{An explicit basis for $\ker(\kappa_{\mathbb{P},f})$} \label{section_roots_and_column_vectors}

Let
\begin{align*} 
R(N, \Sigma) := \{ &\alpha \in M \, | \, \langle \alpha, n(\alpha) \rangle = 1 \textrm{ for some } n(\alpha) \in \Sigma[1] \\
& \textrm{ and } \langle \alpha, n_{j} \rangle \leq 0 \textrm{ for } n_j \in \Sigma[1]\setminus \{n(\alpha) \}  \}
\end{align*}
denote the \textit{roots} of $\Sigma$. Likewise we define $R(N,\Sigma_{C(\Delta)})$ and $R(N,\Sigma_{\Delta})$ by replacing $\Sigma$ by $\Sigma_{C(\Delta)}$ and $\Sigma_{\Delta}$. \\ \\
There are inclusions
\begin{align} \label{inclusions_rays_support_vectors_convex_span}
	\Sigma_{C(\Delta)}[1] \underbrace{\subset}_{(\textrm{since } \Sigma \textrm{ refines } \Sigma_{C(\Delta)})} \Sigma[1] \underbrace{\subset}_{(\textrm{Figure } \ref{figure_support_of_Fine_interior})} \Convhull(\Sigma_{\Delta}[1]).
\end{align}

\begin{lemma} \label{lemma_compare_roots_Sigma_CDelta}
	Let $\Delta$ be an $n$-dimensional lattice polytope with $F(\Delta) \neq \emptyset$. Then
	\[ R(N,\Sigma_{\Delta}) \subset R(N,\Sigma) = R(N,\Sigma_{C(\Delta)}).  \]
\end{lemma}

\begin{proof}
	To the second equality: Let $\alpha \in R(N,\Sigma_{C(\Delta)})$, that is
	\[ \langle \alpha, n(\alpha) \rangle =1, \quad \langle \alpha, n_j \rangle \leq 0 \quad \textrm{ for } n_j \in \Sigma_{C(\Delta)}[1] \setminus \{ n(\alpha) \}. \]
	$\Rightarrow \langle \alpha, n_j \rangle \leq 0$ for $n_j \in \Sigma[1] \setminus \{ n(\alpha) \}$, that is $\alpha \in R(N, \Sigma)$. Conversely assume $\alpha \in R(N,\Sigma)$. If $n_i \notin \Sigma_{C(\Delta)}[1]$ then $\alpha$ would have scalar product $\leq 0$ with all vectors in $\Sigma_{C(\Delta)}[1]$ and thus would be zero since $\Sigma$ refines $\Sigma_{C(\Delta)}$, a contradiction. The first inclusion follows similarly by using (\ref{inclusions_rays_support_vectors_convex_span}).
\end{proof}

We ask for a basis of Laurent polynomials for
\begin{align*}
	\Lie \, \Aut(\mathbb{P}) \subset L(C(\Delta))/ \mathbb{C} \cdot f.
\end{align*}
Remember the results from (\cite{BG99}): Given $f \in B$ there is a map
\begin{align*}
	& \phi_f: T \rightarrow B \\
	& (t_1,t_2,t_3) \mapsto \Big( (x_1,x_2,x_3) \mapsto f(t_1 x_1,t_2 x_2,t_3 x_3) \Big).
\end{align*}
By differentiating $\phi_f$ we get an injective homomorphism $(d \phi_f)_e: \Lie(T) \rightarrow T_{B,f}$ where $e=(1,1,1)$ with
\[ Im(d (\phi_f)_e) = \Big\langle x_1 \cdot \frac{\partial f}{\partial x_1},..., x_3 \cdot \frac{\partial f}{\partial x_3} \Big\rangle.  \]
For $m \in M \cap C(\Delta)$ and $\alpha \in R(N,\Sigma_{C(\Delta)})$ define 
\begin{align}\label{definition_height_function}
	ht_{-\alpha}(m) := \max\{k \in \mathbb{N}_{\geq 0}| \, m - k \cdot \alpha \in C(\Delta)  \}.
\end{align}
Given $\alpha \in R(N,\Sigma_{C(\Delta)})$ we denote by $\Gamma_{-\alpha} \leq C(\Delta)$ the facet to which $n(\alpha)$ is normal.

\begin{remark}
	\normalfont
	Assuming
	\[  \Gamma_{-\alpha} = \{ x \in M_{\mathbb{R}}| \, \langle x,n_{\Gamma} \rangle = b_{\Gamma}  \} \cap C(\Delta) \]
	and $m \in M \cap C(\Delta)$ then
	\begin{align} \label{formula_height_and_scalar_product}
		ht_{-\alpha}(m) =   \langle m,n_{\Gamma} \rangle - b_{\Gamma}.
	\end{align}
\end{remark}

\begin{comment}
\begin{definition} \label{definition_column_vector} (\cite[Def.2.1]{BG99}) \\
	We call a vector $v \in M$ a \textit{column vector} for $C(\Delta)$ if there is a facet $\Gamma_v$ of $C(\Delta)$ such that 
	\[ m+v \in M \cap C(\Delta) \quad \textrm{if } m \in \big( C(\Delta) \cap M \big) \setminus \big(\Gamma_v \cap M \big), \quad \langle v,n_{\Gamma_v} \rangle = -1.  \]
	where $n_{\Gamma_v}$ denotes the inner facet normal to $\Gamma_v$.
\end{definition}

It follows from the definitions that the set $-R(N,\Sigma_{C(\Delta)})$ coincides with the column vectors for $C(\Delta)$.
\end{comment}

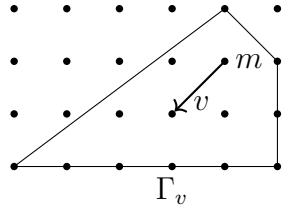
\begin{figure}[H]
	
	\begin{center}
		
		\begin{tikzpicture}[scale=0.7,level/.style={very thick}]

		\fill (-1,1) circle (2pt);
		\fill (-1,2) circle (2pt);	
		\fill (-1,3) circle (2pt);
		\fill (-1,4) circle (2pt);
		\fill (0,1) circle (2pt);
		\fill (0,2) circle (2pt);	
		\fill (0,3) circle (2pt);
		\fill (0,4) circle (2pt);
		\fill (1,1) circle (2pt);
		\fill (1,2) circle (2pt);	
		\fill (1,3) circle (2pt);
		\fill (1,4) circle (2pt);	
		\fill (2,1) circle (2pt);
		\fill (2,2) circle (2pt);	
		\fill (2,3) circle (2pt);
		\fill (2,4) circle (2pt);	
		\fill (3,1) circle (2pt);
		\fill (3,2) circle (2pt);	
		\fill (3,3) circle (2pt);
		\fill (3,4) circle (2pt);	
		\fill (4,1) circle (2pt);
		\fill (4,2) circle (2pt);	
		\fill (4,3) circle (2pt);
		\fill (4,4) circle (2pt);
		
		\draw (-1,1) -- (3,4);
		\draw (3,4) -- (4,3);
		\draw (4,3)  -- (4,1);
		\draw (4,1) -- (-1,1);
		
		\draw[->,thick] (3,3) -- (2.05,2.05);
		\node[below] at (2,1) {$\Gamma_v$}; 
		\node[right] at (2.2,2.2) {$v$};
		\node[right] at (3,3) {$m$};
		
		\end{tikzpicture}
		
	\end{center}
	
	\caption{A lattice polytope with a column vector $v$ and $ht_{v}(m) = 2$.} \label{figure_2_dim_pol_number_of_mod}
\end{figure}

\begin{comment}
There is a natural basis
\[ \Lie(T) = \langle x_1 \cdot \frac{\partial}{\partial x_1},...,x_n \cdot \frac{\partial}{\partial x_n} \rangle.  \]
\leavevmode
\\
For $f \in B$ define a map
\begin{align*}
& \phi_f: T \rightarrow B \\
& (t_1,...,t_n) \mapsto \Big( (x_1,...,x_n) \mapsto f(t_1 x_1,...,t_n x_n) \Big)
\end{align*}
By differentiating $\phi_f$ we get an injective homomorphism $(d \phi_f)_e: \Lie(T) \rightarrow T_{B,f}$ with
\[ Im(d (\phi_f)_e) = \langle x_1 \frac{\partial f}{\partial x_1},..., x_n \frac{\partial f}{\partial x_n} \rangle,  \]
where $e=(1,...,1)$. To verify this formally and to see how the whole space $\Lie \, \Aut(\mathbb{P})$ embeds into $L(C(\Delta))$ we recall some further notations and results from (\cite{BG99}):
\\ \\
For $m \in M \cap C(\Delta)$ and $\alpha \in R(N,\Sigma_{C(\Delta)})$ define 
\[ ht_{-\alpha}(m) := \max\{k \in \mathbb{N}_{\geq 0}| \, m - k \cdot \alpha \in C(\Delta)  \}.  \]

\begin{remark}
	\normalfont
	If 
	\[  \Gamma_{-\alpha} = \{ x \in M_{\mathbb{R}}| \, \langle x,n_{\Gamma} \rangle = b_{\Gamma}  \} \cap C(\Delta), \quad b_{\Gamma} \in \mathbb{Z} \]
	and $m \in M \cap C(\Delta)$ then we obviously have
	\[ ht_{-\alpha}(m) =   \langle n_{\Gamma},m \rangle - b_{\Gamma}. \]
\end{remark}
\end{comment}

Let $S_{C(\Delta)}$ denote the graded semigroup $\mathbb{C}$-algebra over
\[ \Cone(C(\Delta) \times \{1\}) \cap (M \times \mathbb{Z})  \]
Then the function $ht_{-\alpha}$ continues to a map $S_{C(\Delta)} \rightarrow S_{C(\Delta)}$ which respects the grading on $S_{C(\Delta)}$. For $\lambda \in \mathbb{C}$ define a graded automorphism $e_{-\alpha}^{\lambda}: S_{C(\Delta)} \rightarrow S_{C(\Delta)}$ by
\[  e_{-\alpha}^{\lambda}(x^m) := x^m \cdot (1 + \lambda x^{-\alpha})^{ht_{-\alpha}(m)}  \]
%(see \cite[section 3]{BG99}).

\begin{corollar} \label{lemma_basis_lie_algebra_Aut_C(Delta)} (\cite[Lemma 3.1, Thm.3.2b), Thm.5.4]{BG99}) \\
	$\Lie \, \Aut(\mathbb{P})$ has a basis of derivations, which act on $L(C(\Delta))$ as follows
	\begin{align*}
	&  x_i \frac{\partial }{\partial x_i}: \quad x^m \mapsto m_i \cdot x^m, \quad i=1,...,n, \\
	& \frac{\partial e_{-\alpha}^{\lambda}}{\partial \lambda}_{\vert{\lambda = 0}}: \quad   x^m \mapsto ht_{-\alpha}(m) \cdot x^{m-\alpha}, \quad \alpha \in R(N,\Sigma_{C(\Delta)}).
	\end{align*}
\end{corollar}	
\qed

By definition of the tangent sheaf sequence the homomorphism
\[ j: H^0(\mathbb{P}, T_{\mathbb{P}}) \cong H^0(Y, T_{\mathbb{P} \vert{Y}}^{**}) \rightarrow H^0(Y,N_{Y/\mathbb{P}}) \]
is given by applying the derivations from Corollary \ref{lemma_basis_lie_algebra_Aut_C(Delta)} to 
\[ f = \sum\limits_{m \in \Delta \cap M} a_m x^m \in U_{reg}(\Delta) \] 
and restricting to $Y = Y_f$.

\begin{corollar} \label{theorem_basis_ker_kappa_P_subset_LCDelta}
	Given the conditions of Theorem \ref{theorem_number_of_moduli} $\ker(\kappa_{\mathbb{P},f})$ has the basis
	\begin{align*}
	& \quad \quad \quad \quad \quad \quad x_1 \cdot \frac{\partial f}{\partial x_1},..., x_n \cdot \frac{\partial f}{\partial x_n}, \\
	& w_{-\alpha}(f) := \sum\limits_{m \in \Delta \cap M} ht_{-\alpha}(m) \cdot a_m \cdot x^{m-\alpha}, \quad \alpha \in  R(N,\Sigma_{C(\Delta)}).
	\end{align*}
\end{corollar}

\begin{figure}[H]
	
	\begin{center}
		
		\begin{tikzpicture}[scale=0.7,level/.style={very thick}]
			
			\begin{scope}[xshift=-7cm]
				
				\draw (-1,1) circle(2pt);
				\draw (-1,2) circle(2pt);
				\draw (-1,3) circle (2pt);
				\draw (-1,4) circle (2pt);
				\draw (0,1) circle(2pt);
				\fill (0,2) circle (2pt);	
				\fill (0,3) circle (2pt);
				\fill (0,4) circle (2pt);
				\draw (1,1) circle(2pt);
				\fill (1,2) circle (2pt);	
				\fill (1,3) circle (2pt);
				\draw (1,4) circle (2pt);
				\draw (2,1) circle(2pt);
				\fill (2,2) circle (2pt);	
				\draw (2,3) circle (2pt);
				\draw (2,4) circle (2pt);
				\draw (3,1) circle(2pt);
				\draw (3,2) circle (2pt);	
				\draw (3,3) circle (2pt);
				\draw (3,4) circle (2pt);	
				
				\node[right] at (0,2) {\tiny $1$};
				\node[right] at (1,2) {\tiny $1$};
				\node[right] at (2,2) {\tiny $1$};
				\node[right] at (0,3) {\tiny $2$};
				\node[right] at (1,3) {\tiny $2$};
				\node[right] at (0,4) {\tiny $3$};

				\draw[dashed] (1/2-1,1/2+1) -- (9/2-1/4-1,1/4+1/4+1);
				\draw[dashed] (1/2-1,1/2+1) -- (1-1,3.5+1);
				\draw[dashed] (1-1,3.5+1) -- (9/2-1/4-1,1/4+1/4+1);

				\draw (-1,1) -- (3,1);
				\draw (3,1) -- (0,4);
				\draw (0,4)  -- (-1,1);
				
				\draw[->] (0,3) -- (0.95,2.05);
				\node[below] at (0.2,0.8) {$\Gamma_{-\alpha}$}; 
				\node[right] at (0.3,2.6) {$-\alpha$};
			\end{scope}
			
			\draw[->] (-3,3) -- (-2,3);

			\draw (-1,1) circle(2pt);
			\draw (-1,2) circle(2pt);
			\draw (-1,3) circle (2pt);
			\draw (-1,4) circle (2pt);
			\draw (0,1) circle(2pt);
			\draw (0,2) circle (2pt);	
			\draw (0,3) circle (2pt);
			\draw (0,4) circle (2pt);
			\fill (1,1) circle(2pt);
			\fill (1,2) circle (2pt);	
			\fill (1,3) circle (2pt);
			\draw (1,4) circle (2pt);
			\fill (2,1) circle(2pt);
			\fill (2,2) circle (2pt);	
			\draw (2,3) circle (2pt);
			\draw (2,4) circle (2pt);
			\fill (3,1) circle(2pt);
			\draw (3,2) circle (2pt);	
			\draw (3,3) circle (2pt);
			\draw (3,4) circle (2pt);	
			
			\draw[dashed] (1/2,1/2) to (9/2-1/4,1/4+1/4);
			\draw[dashed] (1/2,1/2) -- (1,3.5);
			\draw[dashed] (1,3.5) -- (9/2-1/4,1/4+1/4);

			\draw (-1,1) -- (3,1);
			\draw (3,1) -- (0,4);
			\draw (0,4)  -- (-1,1);
			
			%\draw[->] (0,3) -- (0.95,2.05);
			%\node[below] at (0.2,0.8) {$\Gamma_{-\alpha}$}; 
			%\node[right] at (-0.6,2.3) {$-\alpha$};
			
		\end{tikzpicture}
		
	\end{center}
	\caption{On the left: The column vector $-\alpha$ with facet $\Gamma_{-\alpha}$ and all lattice points $m \in C(\Delta)$ with $ht_{-\alpha}(m) > 0$. On the right: $w_{-\alpha}(f)$ has support on the thick lattice points.} \label{figure_support_of_w_-alpha}
\end{figure}
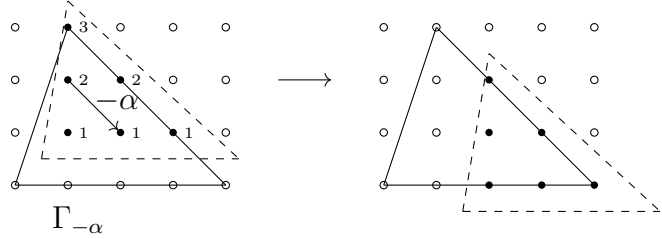

\begin{example} \label{example_proj_hyp_Griffiths_jac_ideal_ker_Kod_Spenc}
	\normalfont
	If 
	\[ \Delta = d \cdot \Delta_n, \quad f \in U_{reg}(\Delta),  \]
	then 
	\[ C(\Delta) = \Delta, \quad \Sigma = \Sigma_{\Delta}, \quad \mathbb{P} = \mathbb{P}^n  \]
	and $Y_f$ is a smooth degree $d$ hypersurface in $\mathbb{P}^n$.	For such an hypersurface it is shown in (\cite[Lemma 6.15]{Voi03}) that
	\begin{align} \label{proj_hyp_degree_d_ker_kappa}
	\ker(\kappa_f) \cong J_{f,griff}^d,
	\end{align}
	if we work with the family $\mathcal{X} \rightarrow U_{reg}(\Delta)$ (if we projectivize then we have to mod out $f$ from the kernel). Here $J_{f,griff}^d$ denotes the $d$-th homogeneous component of \textit{Griffiths Jacobian ideal}
	\[ J_{f,griff} := (\frac{\partial f}{\partial x_0},...,\frac{\partial f}{\partial x_n}) \unlhd \mathbb{C}[x_0,...,x_n].  \]
	The roots of $\Sigma$ are given by 
	\[ \pm e_i, \quad i=1,...,n, \quad   \pm e_i \mp e_j, \quad i,j=1,...,n, \quad i \neq j  \]
	and if $d \geq n+1$ then Theorem \ref{theorem_basis_ker_kappa_P_subset_LCDelta} restricts to the result (\ref{proj_hyp_degree_d_ker_kappa}) up to homogenization.
\end{example}

\section{An explicit basis for $\ker(\kappa_{f})$} \label{section_number_of_mod_general_case}

\begin{theorem} \label{theorem_numb_mod_not_can_closed}
	Given the conditions of Theorem \ref{theorem_number_of_moduli} $\ker(\kappa_{f})$ has the basis
	\begin{align*}
	x_i \frac{\partial f}{\partial x_i}, \quad i=1,...,n, \quad w_{-\alpha}(f), \quad \alpha \in  R(N,\Sigma_{\Delta}).
	\end{align*}
\end{theorem}

We first prove the following Proposition which reduces the proof to a combinatorial argument

\begin{proposition} \label{proposition_ker_kappa_f_intersection_with_Newton_pol}
	Let $\Delta$ be an $n$-dimensional lattice polytope with $F(\Delta) \neq \emptyset$. Then $\kappa_f$ equals the restriction of $\kappa_{\mathbb{P},f}$ to $L(\Delta)/ \mathbb{C} \cdot f$ and thus
	\begin{align} 
	\ker(\kappa_{f}) \cong  \ker(\kappa_{\mathbb{P},f}) \cap L(\Delta).
	\end{align} 		
\end{proposition}

\begin{proof} \label{construction_w_alpha_ker_kappa}
	\normalfont
	The following reduction step is similar to (\cite[Ch.2.1]{Koe91} and \cite[Lemma 6.15]{Voi03}):	Let $V \subset \mathbb{P}$ be the union of all torus orbits of dimension $\geq n-2$. Then $U = U_f:= V \cap Y_f$ is smooth and $\codim_{Y_f}(Y_f \setminus U_f) \geq 2$ for every $f \in B$. Let 
	\[ W := (V \times B) \cap \mathcal{X}.  \]
	Consider the differential 
	\[ (pr_1)_*: T_{W \vert{U}} \rightarrow T_{V \vert{U}}  \]
	of the first projection. All the sheaves we consider are reflexive, therefore there is no difference in working with $Y$, $\mathcal{X}$ and $\mathbb{P}$. $pr_1$ restricts to an isomorphism
	\[ Y_f \times \{f\} \rightarrow Y_f,  \]
	thus $(pr_1)_*$ restricts to the identity on $T_Y$. The map
	\[ (pr_1)_*: N_{Y/\mathcal{X}} \cong H^0(Y,N_{Y/\mathcal{X}}) \otimes \mathcal{O}_Y \subset H^0(Y,N_{Y/\mathbb{P}}) \otimes \mathcal{O}_Y \rightarrow   N_{Y/\mathbb{P}} \]
	is given by multiplication of sections. In effect we obtain a commutative diagram
	\[  
	\begin{tikzcd}[scale cd = 0.9]
	0 \arrow[r] & H^0(Y,T_Y) \arrow[r] \arrow{d}{id} & H^{0}(Y,T_{\mathcal{X} \vert{Y}}^{**}) \arrow[r] \arrow[swap]{d}{(pr_1)_*} & H^{0}(Y, N_{Y/\mathcal{X}}) \arrow{d} \arrow{r}{\kappa_f} & \Ext_{\mathcal{O}_Y}^1(\Omega_Y^1, \mathcal{O}_Y)  \arrow{d}{id} \\
	0 \arrow[r] & H^0(Y,T_Y) \arrow[r] & H^{0}(Y, T_{\mathbb{P} \vert{Y}}^{**}) \arrow[r] & H^{0}(Y, N_{Y/\mathbb{P}}) \arrow{r}{\kappa_{\mathbb{P},f}} & \Ext_{\mathcal{O}_Y}^1(\Omega_Y^1, \mathcal{O}_Y)
	\end{tikzcd}
	\]	
	and
	\[ \ker(\kappa_f) \cong H^0(Y,T_{\mathcal{X} \vert{Y}}^{**}) \cong H^0(Y,T_{\mathbb{P} \vert{Y}}^{**}) \cap L(\Delta) \cong \ker(\kappa_{\mathbb{P},f}) \cap L(\Delta). \]
\end{proof}

The $x_i \frac{\partial f}{\partial x_i}$ obviously belong to $L(\Delta)$ but the $w_{-\alpha}(f)$ need not have support on $\Delta$ as the folowing example shows:

\begin{example} \label{example_elliptic_surface_not_can_closed}
	\normalfont
	Consider the polytope
	\begin{align*}
		\Delta = \langle \begin{pmatrix} -1\\-1\\-1 \end{pmatrix},  \begin{pmatrix}  5\\1\\3 \end{pmatrix}, \begin{pmatrix} -1\\10\\0 \end{pmatrix}, \begin{pmatrix} -1\\-1\\0 \end{pmatrix}  \rangle
	\end{align*}	
	$\Delta$ has $3$ interior lattice points, $F(\Delta)$ is $1$-dimensional and $C(\Delta)$ has the additional vertex $(1,-1,1)$. We obtain a family of elliptic surfaces $\mathcal{X} \rightarrow B$. There are $7$ roots
	\[ R(N,\Sigma) = \{ \begin{pmatrix} -3\\-1\\-2 \end{pmatrix},  \begin{pmatrix}  -1\\-4\\-1 \end{pmatrix}, \begin{pmatrix} -1\\-3\\-1 \end{pmatrix}, \begin{pmatrix} -1\\-2\\-1 \end{pmatrix}, \begin{pmatrix} -1\\-1\\-1 \end{pmatrix}, \begin{pmatrix} -1\\0\\-1 \end{pmatrix}, \begin{pmatrix} 0\\-1\\0 \end{pmatrix}  \}  \]
	and one root $(-1,0,-1)$ not belonging to $R(N,\Sigma_{\Delta})$. The column vector $-\alpha = (1,0,1)$ belongs to the facet
	\[ \Gamma_{-\alpha} = \langle  \begin{pmatrix}  5\\1\\3 \end{pmatrix}, \begin{pmatrix} -1\\10\\0 \end{pmatrix}, \begin{pmatrix} -1\\-1\\0 \end{pmatrix},\begin{pmatrix} 1\\-1\\1 \end{pmatrix}  \rangle \] 
	of $C(\Delta)$. The vertex $(-1,-1,-1)$ does not lie on $\Gamma_{-\alpha}$ and
	\[ -\alpha + (-1,-1,-1) = (0,-1,0) \notin \Delta.  \]
	Thus only $6$ of the roots in $R(N,\Sigma)$ reduce the number of moduli.
\end{example}

\textit{(Proof of Theorem \ref{theorem_numb_mod_not_can_closed})} \\ \\
	The proof is rather technical. By Proposition \ref{proposition_ker_kappa_f_intersection_with_Newton_pol} 
	\[ \ker(\kappa_f) = \ker(\kappa_{\mathbb{P},f}) \cap L(\Delta)   \]
	and $R(N,\Sigma_{\Delta}) \subset R(N,\Sigma_{C(\Delta)})$ by Lemma \ref{lemma_compare_roots_Sigma_CDelta}. Let $R:= R(N,\Sigma_{C(\Delta)}) \setminus R(N,\Sigma_{\Delta})$. The Theorem is a consequence of the three points below.
	\begin{itemize}
		\item $\alpha \in R(N,\Sigma_{\Delta}) \Rightarrow w_{-\alpha}(f) \in L(\Delta)$.
		\item $\alpha \in R \Rightarrow w_{-\alpha}(f) \notin L(\Delta)$.
		\item Varying $\alpha \in R$ the $w_{-\alpha}(f)$ are linearly independent in $L(C(\Delta))/L(\Delta)$.
	\end{itemize}
	The necessity of the first two points is obvious and the last point assures that no linear combination of the $w_{-\alpha}(f)$, where $\alpha \in R$, lies in $\ker(\kappa_f)$.

	\leavevmode
	\\
	\underline{First point:} To $\alpha \in R(N,\Sigma_{\Delta})$ is associated both $\Gamma_{-\alpha} \leq \Delta$ and $\Gamma_{-\alpha}' \leq C(\Delta)$. We show
	\begin{align} \label{thm_nofmodnotcancl_help_formula}
		\Gamma_{-\alpha} \subset \Gamma_{-\alpha}',
	\end{align}
	since then for $m \in M \cap \Delta$, $m \notin \Gamma_{-\alpha}$ we get $m - \alpha \in \Delta$, that is $w_{-\alpha}(f) \in L(\Delta)$. Concerning (\ref{thm_nofmodnotcancl_help_formula}): Given $n_i \in \Sigma_{\Delta}[1]$ with $\langle \alpha, n_i \rangle = 1$ and $n_j \in \Sigma_{C(\Delta)}[1]$ with $\langle \alpha, n_j \rangle =1$ then $n_i = n_j$ by (\ref{inclusions_rays_support_vectors_convex_span}). It follows $\Gamma_{-\alpha} \subset \Gamma_{-\alpha}'$ since
	\[ \Min_{C(\Delta)}(n_i) = \Min_{\Delta}(n_i). \]
	Thus $\Gamma_{-\alpha} \subset \Gamma_{-\alpha}'$.  \\ \\
	\underline{Second point:} There is a facet $\Gamma_{-\alpha}$ of $C(\Delta)$ such that
	\[ m - \alpha \in C(\Delta) \quad \textrm{for } m \in C(\Delta) \cap M, \quad m \notin \Gamma_{-\alpha}. \]
	First assume that $\Gamma_{-\alpha} \cap \Delta$ is also a facet of $\Delta$. There is 
	$n_j \in \Sigma_{\Delta}[1] \setminus \{n_{\Gamma_{-\alpha}} \}$ with $\langle \alpha, n_j \rangle > 0$ since $\alpha \notin R(N,\Sigma_{\Delta})$. Given $m \in Vert(\Gamma_j)$, then $m \in \Supp(f)$ and $m-\alpha \notin \Delta$ since 
	\[ \langle m - \alpha, n_j \rangle < \Min_{\Delta}(n_j).  \]
	$\Rightarrow w_{-\alpha}(f) \notin L(\Delta)$. Assume that $\Gamma_{-\alpha} \cap \Delta$ is a face of $\Delta$ of dimension $< n-1$. The convex span
	\[ \langle m \in Vert(\Delta) \, | \, m - \alpha \notin \Delta \rangle  \]
	is of dimension $\geq n-1$. $\Rightarrow$ there is $m \in Vert(\Delta)$ with 
	\[  \quad m \notin \Gamma_{-\alpha}, \quad m - \alpha \notin \Delta, \]
	that is $w_{-\alpha}(f) \notin L(\Delta)$.
	\leavevmode
	\\ \\
	\underline{Third point:} Given a fixed facet $\Gamma = \Gamma_{-\alpha}$ of $C(\Delta)$ all 
	\[ \alpha \in R(N,\Sigma_{C(\Delta)}) \setminus R(N,\Sigma_{\Delta})  \]
	with $\Gamma_{-\alpha} = \Gamma$ build the lattice points on a lattice polytope $P \subset M_{\mathbb{R}}$. \\
	Given $\alpha \in Vert(P)$ there is $m \in \Supp(f)$ such that $x^{m-\alpha}$ does not appear in the support of any other $w_{-\alpha'}(f)$. Thus $w_{-\alpha}(f)$ does not appear with nonzero coefficient in any relation between the $w_{-\alpha'}(f)$. We then break down $P$ vertex by vertex. \\ \\
	Let $\Gamma_1,\Gamma_2$ be two different facets of $C(\Delta)$ and $\alpha_1,\alpha_2\in R(N,\Sigma_{C(\Delta)}) \setminus R(N,\Sigma_{\Delta})$ roots to these facets. Given a relation in
	\[ L(C(\Delta))/L(\Delta) \]
	in which both $w_{-\alpha_1}(f)$ and $w_{-\alpha_2}(f)$ appear with nonzero coefficients there is $v \in \Supp(f)$ with
	\[ \langle v-\alpha_1,n_1 \rangle < \Min_{\Delta}(n_1), \quad v-\alpha_1+ \alpha_2 \in M \cap \Delta.  \]
	Then
	\[ \langle v- \alpha_1+\alpha_2, n_1 \rangle \geq \Min_{\Delta}(n_1),  \]
	but $\langle \alpha_2,n_1 \rangle \leq 0$ since $\alpha_2$ is a root for $n_2 \neq n_1$, a contradiction.	
\qed	
	
\begin{remark}
	\normalfont
	Given a common toric resolution of singularities
	
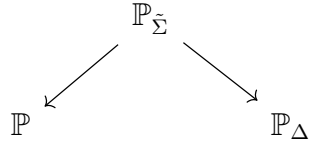
\[\begin{tikzcd}
	& \mathbb{P}_{\tilde{\Sigma}} \arrow[swap]{dl} \arrow{dr} \\
	\mathbb{P} &&  \mathbb{P}_{\Delta}
	\end{tikzcd}
	\]
	there is a deformation of smooth toric hypersurfaces $\tilde{X} \rightarrow B$. Take the Kodaira-Spencer map $\tilde{\kappa}_f$. Then $R(N,\tilde{\Sigma}) \subset R(N,\Sigma_{\Delta})$ and $\ker(\tilde{\kappa}_f)$ is gotten as $\ker(\kappa_f)$ but with $\alpha \in R(N,\tilde{\Sigma})$ instead of $\alpha \in R(N,\Sigma_{\Delta})$.
\end{remark}

%\section{The vanishing of $H^0(Y,T_Y)$} \label{section_number_of_mod_general_case}

\section{The number of moduli for subfamilies} \label{section_subfamilies}

\begin{remark} \label{remark_arbitrary_subfamilies_number_of_moduli}
	\normalfont
	Let $\Delta$ be an $n$-dimensional polytope with $F(\Delta) \neq \emptyset$ and $A \subset M \cap \Delta$ a subset containing all vertices of $\Delta$. Let
	\[ f:=   \sum\limits_{m \in A} a_m x^{m}.  \]
	Then  for $(a_m)_{m \in A}$ generic $f$ is nondegenerate with respect to $\Delta$ (see \cite[Ch.10]{GKZ94} and \cite[Prop.2.16]{Bat03}). Denote the resulting open subset of $\mathbb{C}^{|A|}$ by $U_A$ and the restriction of $\mathcal{X}$ to $U_A$ by $\mathcal{X}_A$. 
	%We obtain
	%\[ T_{U_A,f} \cong \bigoplus\limits_{m \in A} \mathbb{C} \cdot x^m.  \]
	Taking the quotient by the Laurent polynomials
	\[ f, \quad x_i \frac{\partial f}{\partial x_i} \quad i=1,...,n \]
	will reduce the number of moduli of the subfamily $\mathcal{X}_A \rightarrow U_A$ by $n+1$. 
	\begin{comment}
		This is because the matrix which has as columns the vectors $(1,m)$ where $m \in A$ has rank $n+1$ and after multiplying the $k$-th column by $a_{m_k}$ it still has rank $n+1$. But this means that there is no relation between
		\[ f, \, x_1 \frac{\partial f}{\partial x_1}, ..., x_n \frac{\partial f}{\partial x_n}.  \]
	\end{comment}
	Concerning the $w_{-\alpha}(f)$ it seems to be difficult to decide in general if there are $c_{\alpha} \in \mathbb{C}$ with
	\[ \sum\limits_{\alpha} c_{\alpha} w_{-\alpha}(f) \in T_{U_A,f} \setminus \{0\}.  \]
	Therefore we restrict to some special cases:
\end{remark}	

\begin{example} \label{example_subfamily_to_monomials}
	\normalfont
	Assume that $\Delta$ is an $n$-dimensional simplex with $F(\Delta) \neq \emptyset$ and $A$ equals the set of vertices of $\Delta$. Then by varying the coefficients to $A$ we obtain a family with $\kappa = 0$. This generalizes:
\end{example}

\begin{lemma} \label{lemma_number_of_moduli_subfam_to_vertices}
	Given the conditions of Theorem \ref{theorem_number_of_moduli} let $A$ denote the set of vertices of $\Delta$. Assume that for every facet
	\[ \Gamma = \Delta \cap \{ x \in M_{\mathbb{R}}| \, \langle x,n_{\Gamma} \rangle = b_{\Gamma} \}  \]
	with $n_{\Gamma} \in \Sigma_{\Delta}[1]$ there is no vertex of $\Delta$ lying in the plane
	\[ H_{\Gamma,+1} :=  \{ x \in M_{\mathbb{R}}| \, \langle x,n_{\Gamma} \rangle = b_{\Gamma} + 1 \}  \]
	Then the subfamily to $A$ has
	\[ \# \{\textrm{vertices}\} - n-1  \]
	number of moduli.
\end{lemma}

\begin{proof}
	Given a root $\alpha$ and a lattice point $m$ on $\Delta$ then $m-\alpha$ lies exactly one step closer to the facet $\Gamma_{-\alpha}$ (and not closer to any other facet). Assume that $m$ and $m-\alpha$ are vertices of $\Delta$. Then $m-\alpha$ lies on $\Gamma_{-\alpha}$ since else $m, m-2 \alpha \in \Delta$ and $m-\alpha$ would not be a vertex. But then $m \in H_{\Gamma_{-\alpha},+1}$ contradicting the assumption.
\end{proof}	

\begin{comment}
\begin{remark}
	\normalfont
	Assume $A \subset \Delta \cap M$ contains all vertices of $\Delta$ and further
	\begin{itemize}
		\item $A$ contains at least one interior lattice point of $\Delta$
		\item A does not contain an interior lattice point on any facet of $\Delta$.
	\end{itemize}
	Then if 
	\[ f = \sum\limits_{m \in A} a_m x^m  \]
	and $a_m \neq 0$ for all $m \in A$ then
	\[ \dim \, Im(\kappa_f) = \#A - n -1.  \]
	The reason is that for $m \in L^*(\Delta)$ and $\alpha \in R(N,\Sigma_{\Delta})$ either $m-\alpha \in L^*(\Delta)$ or $m-\alpha \in L^*(\Gamma_{-\alpha})$. Thus we find an $m \in \Supp(f)$ with $x^{m-\alpha} \notin T_{U_A,f}$ and it follows that $w_{-\alpha}(f) \notin T_{U_A,f}$.
\end{remark}

	\begin{example} 
		\normalfont
		In the family of quartic surfaces in $\mathbb{P}^3$
		\[ t_0 x_0^4 + t_1 x_1^4 + t_2 x_2^4 + t_3 x_3^4 = 0  \]
		all smooth members are isomorphic, but if we add $t_4 x_0x_1x_2x_3$ then this becomes a one-dimensional family, essentially the \textit{Dwork family}.
	\end{example}
\end{comment}

\end{document}